\documentclass{article}
\usepackage{amsmath,amssymb,graphicx} 
\usepackage{amssymb}
\usepackage{amsthm}
\usepackage[margin=0.9in]{geometry} 
 \newtheorem{theorem}{Theorem}[section]
 \newtheorem{definition}{Definition}[section]
 \newtheorem{lemma}{Lemma}[section]
\newtheorem{remark}{Remark}[section]
\newtheorem{preposition}{Preposition}[section]
 \usepackage{hyperref}
 \usepackage{tikz}
 \allowdisplaybreaks
 \usepackage{comment} 
 \usepackage{enumitem}  
\begin{document}

\nocite{*} 

\title{Analysis of a fully discrete approximation to a moving-boundary problem  describing rubber exposed to diffusants}

 \author{Surendra\,Nepal $^{1,*}$,  Yosief\,Wondmagegne$^1$, Adrian\,Muntean$^1$ \\
$^1$ Department of Mathematics and Computer Science, Karlstad University, Sweden\\
*surendra.nepal@kau.se}
\date{\today} 
\maketitle

\noindent

\begin{abstract}

We present a fully discrete scheme for the numerical approximation of a moving-boundary problem describing diffusants penetration into rubber. Our scheme utilizes the Galerkin finite element method for the space discretization combined with the backward Euler method for the time discretization. Besides dealing with the existence and uniqueness of solution to the fully discrete problem, we derive a \textit{a priori} error estimates for the mass concentration of the diffusants, and respectively, for the position of the moving boundary. 
Numerical illustrations verify the  obtained theoretical order of convergence in physical parameter regimes.


	
\end{abstract}
\section{Introduction}

We study the fully discrete approximation of a one-dimensional moving-boundary problem describing the penetration of diffusants into rubber. The model presented here was proposed recently in \cite{nepal2021moving}, where the simulation output was compared to experimental data. In this framework, our interest is focused on the numerical analysis of the model. Relying on previous mathematical analysis work done for  an adsorption model with moving swelling interfaces (see \cite{kumazaki2020global}), which shares the structure of the equations with our current moving-boundary model,   we have provided in \cite{nepal2021error} an analysis of the control of the errors produced by a FEM semi-discretization of our model equations. In this paper, we turn our attention to estimating the errors produced  by combining time and space discretizations.  Our analysis of the fully discrete approximation to the model equations relies on our previous results \cite{kumazaki2020global,nepal2021error} and should be seen as a natural continuation of the work. Browsing the existing literature, one can find a lot of information regarding the rigorous error analysis of semi-discrete approximation of free- and moving-boundary problems. However, much less seems to be known what concerns the analysis of fully discrete approximation schemes even for one dimensional formulations where the moving interface is in fact only a moving point (with {\em a priori} unknown location). The references \cite{ahn2003error,lee2006error} were particularly useful for our investigation. In \cite{lee2006error} H. Y. Lee develops a fully discrete scheme for a Stefan  problem with non-linear free boundary condition and  investigates the order of convergence of the scheme. In ref. \cite{ahn2003error}, the authors construct and analyze fully discrete methods  for a free boundary problem arising in the polymer technology. By using the Galerkin finite element  formulation in space and a backward Euler scheme in time, the authors were able to prove the \textit{a priori} error estimate for the concentration of the solvent and for the position of  moving boundary. At the technical level, we were very much inspired by the technique that has been used in \cite{ahn2003error} to get the \textit{a priori} error bound.  It is also worth mentioning that the main difference between the problem considered in these papers and our problem lies in  the
choice of the boundary conditions. Both of the cited papers impose at least 
an homogeneous Dirichlet boundary condition at one of the boundaries, while we
impose flux boundary conditions at both boundaries that bring in boundary terms that need a careful handling. \\

The problem setting we are studying here is as follows: For a fixed given observation time $T_f\in (0, \infty)$, let the interval $[0, T_f]$ be the time span of the physical processes we are considering.  Let  $x\in [0, s(t)]$ and $t \in [0, T_f]$ denote the space and respectively, the time variable. Let $m(t, x)$  be the concentration of diffusant placed in position $x$ at time  $t$.
 The diffusants concentration  $m(t, x)$  acts in the region $Q_s(T_{f})$ defined by $$ Q_s(T_{f}):= \{ (t, x) | t \in (0, T_{f}) \; \text{and}\; x \in (0, s(t))\}.$$ 
 The problem reads: Find 
 $m(t, x)$ together with the position of the moving boundary (interface) $x = s(t)$ for $t\in(0, T_f)$ such that the couple $(m(t, x), s(t))$ satisfies the following evolution problem: 
\begin{align}
\label{a11}&\displaystyle \frac{\partial m}{\partial t} -D \frac{\partial^2 m}{\partial x^2} = 0\;\;\; \ \text{in}\;\;\; Q_s(T_{f}),\\
\label{a12}
&-D \frac{\partial m}{\partial x}(t, 0) = \beta(b(t) -\text{H}m(t, 0))  \;\;\; \text{for}\;\;  t\in(0, T_{f}),\\
\label{a13}&-D \frac{\partial m}{\partial x}(t, s(t))  =s^{\prime}(t)m(t, s(t))  \;\;\; \text{for}\;\; t\in(0, T_{f}),\\
\label{a14}&s^{\prime}(t) = a_0 (m(t, s(t)) - \sigma(s(t)) \;\;\;\;\text{for } \;\;\; t \in (0,T_{f}),\\
\label{a15}&m(0, x) = m_0(x) \;\;\;\text{for}\;\;\; x \in [0, s(0)],\\
\label{a16} & s(0) = s_0>0\; \text{with}\;\; 0<  s_0< s(t) < L, 
\end{align}
where  $a_0>0$ is a  kinetic coefficient, $\beta$ is a positive constant describing the capacity of the interface at $x=0$,  $D>0$ is the effective diffusion constant, $\text{H}>0$ is the Henry's constant. Additionally, $\sigma$ is a real function, $b$ is a given boundary concentration on $[0, T]$,
$s_0>0$ is the initial position of the moving boundary, while $m_0$ represents the initial concentration of the diffusant. 

This model reminds of the work by Astarita and collaborators (compare \cite{astaluta1978class} and follow-up papers) on free boundary problems posed in the context of  polymeric materials; see \cite{crank1984free,fasano1986problem,murray1988finite,conrad1990well} for classical older works on the topic.   We refer the reader to \cite{vcanic2021moving} to a very recent collection of   modeling, mathematical analysis, and numerical simulation aspects of moving-boundary problems arising in {\em fluid-structure} interaction scenarios. It is worthwhile to note that in our model the effect of the {\em structure} part, i.e. either rubber's mechanics (see e.g. \cite{aiki2020macro}) and/or material's capacity to perform capillary transport (see e.g. \cite{lunowa2021dynamic}), is incorporated in the shape of the nonlinearity $\sigma(\cdot)$. 

The paper has the following structure: In Section \ref{preliminaries} we specify the used notation,  technical assumptions, as well as  a couple of  of useful basic inequalities. The weak formulation of our model together with the FEM discretization in space are included in Section \ref{transformation}. We recall here also some results obtained earlier by us concerning the semi-discrete FEM approximation. The bulk of the paper is the error analysis of the fully discrete approximation of our concept of solution. This is the  purpose of Section \ref{fully}.   Section \ref{experiments} contains a couple of numerical experiments confirming the theoretical convergence rates.  Our conclusions on the obtained estimates and ideas for further work for are listed in Section \ref{conclusion}.

\section{Notation. Basic inequalities. Technical assumptions} \label{preliminaries}

 In this framework, standard notations for Sobolev and Bochner spaces are used. An introduction to  Sobolev and Bochner spaces as well as the usual notations, definition of norms and inner products can be found, for instance, in  \cite{adams2003sobolev, kufner1977function}. For the convenience of writing, we denote by $\|\cdot\|$ and $(\cdot,\cdot)$ the norm, and respectively, the inner product in  $L^2(\Omega)$.  Furthermore, $\|\cdot\|_\infty$ refers to the norm of $L^\infty(\Omega)$. We  also use the notation $(^\prime)$ to indicate the derivative with respect to time variable.\\
For the benefit of reader,  we collect a few elementary inequalities that we frequently use in this context.
\begin{enumerate}[label = (\roman*)]
	\item Young's inequality:
	\begin{align}
	\label{a3}ab \leq \xi a^2 + c_{\xi} b^2,
	\end{align}
	where $a, b \in \mathbb{R}_{+},\; \xi >0,\; c_{\xi}:= \displaystyle \frac{1}{4\xi} >0.$
	\item Interpolation inequality:  For all $u \in H^1(0, 1)$,
		it exists a constant $\hat{c}>0$ depending on $\theta \in [\frac{1}{2}, 1)$ such that 
		\begin{align}
		\label{a33}
		\| u\|_{\infty} \leq \hat{c} \|u\|^{\theta} \left\Vert u\right\Vert_{H^1(0,1)}^{1 - \theta}.
		\end{align}
		For  $\theta = 1/2$, one gets 
		\begin{align*}
		\| u\|_{\infty}^2\leq \hat{c}\left(\xi \left\Vert \frac{\partial u}{\partial y}\right\Vert^2 + ( \xi + c_{\xi}) \|u\|^2 \right),
		\end{align*} 
		where $\xi$ and $c_{\xi}$ are as in \eqref{a3}; see details in \cite{zeidlernonlinear} p. 285 (example 21.62). Note that for $\phi \neq \Omega \subset \mathbb{R}^d$ a bounded domain, \eqref{a33} is related to the so-called Agmon's inequality, i.e. for all $u \in H^{s_2}(\Omega)$ it exists $\hat{c}>0$ such that 
		\begin{align*}
		    \| u\|_{\infty} \leq \hat{c} \| u\|^{\theta}_{H^{s_1}(\Omega)}  \| u\|^{1- \theta}_{H^{s_2}(\Omega)}
		\end{align*}
		for $0< \theta <1,\; s_1 < d/2 < s_2,\; d/2 = \theta s_1 + (1-\theta)s_2$. Taking here $s_1=0, s_2 = 1$ and $d=1$, we are led to $\eqref{a33}$.
\end{enumerate}
Throughout this paper, the involved parameters are assumed to fulfill the following conditions: 
\begin{enumerate}[ label = ({A}{\arabic*})]
	\item \label{A1} 	$a_0,\;\text{H},\; D,\; s_0, \; T_{f}$ are positive constants.
	\item  	\label{A2}$b \in W^{1, 2}(0, T_{f})$ with $0< b_{*} \leq b \leq b^{*}$ on $(0, T_{f})$, where $b_{*}$ and $b^{*}$ are positive constants.
	\item 	\label{A3} $\beta \in C^1(\mathbb{R}) \cap W^{1, \infty} (\mathbb{R})$ such that $\beta = 0$ on $(\infty, 0]$, and there exists $r_{\beta}>0$ such that $\beta^{\prime}>0$ on $(0, r_{\beta})$ and $\beta = k_0$ on $[r_{\beta}, +\infty)$, where $k_0>0$.
	\item 	\label{A4} $\sigma \in C^1 (\mathbb{R}) \cap W^{1, \infty} (\mathbb{R})$ such that $\sigma = 0$ on $ (-\infty, 0 )$, and  there exists  $r_{\sigma}$ such that $\sigma^{\prime}>0$ on $(0, r_{\sigma})$ and $\sigma = c_0$ on $[r_{\sigma}, +\infty),$ where $c_0$  satisfies 
	\begin{align}
	0 <c_0 < \min\{ 2 \sigma(0), b^*\text{H}^{-1}\}.
	\end{align}
	\item 	\label{A5}$0<s_0< r_{\sigma}$ and $m_0 \in H^1(0, s_0)$ such that $\sigma(0) \leq u_0 \leq b^*\text{H}^{-1}$ on $[0, s_0].$
\end{enumerate}
The assumptions \ref{A1}--\ref{A5} are adopted from \cite{kumazaki2020global}. (A1) and (A2) have a clear physical meaning, while (A3)-(A5) are of pure technical nature. They delimit a framework where the solvability of our moving-boundary problem is guaranteed. \\

\section{Weak formulation. Galerkin approximation. Preliminary results}
\label{transformation}
As introduced in \cite{nepal2021moving, nepal2021error}, after the non-dimensionalization   and  transformation $y = x/s(t)$, the  problem \eqref{a11}--\eqref{a16} transforms into the following problem in fixed domain  $Q(T) := \{ (\tau, y) |\; \tau \in (0, T) \; \text{and}\; y \in (0, 1)\}$.
\begin{align}
\label{aa17} 
&\displaystyle\frac{\partial{u}}{\partial{\tau}} - y\frac{h^{\prime}(\tau)}{h(\tau)}  \frac{\partial{u}}{\partial{y}}- \frac{1}{(h(\tau))^2}\frac{\partial^2{u}}{\partial{y^2}} = 0
\;\;\; \ \text{in}\;\;\;Q(T),\\
\label{aa18} &- \frac{1}{h(\tau)} \frac{\partial u}{\partial y}(\tau, 0) = \text{\rm{Bi}}\left(\frac{b(\tau)}{m_{0}}- \text{H}u(\tau, 0)\right)\;\;\; \text{for}\;\; \tau\in(0, T), \\
\label{aa19}&- \frac{1}{h(\tau)} \frac{\partial u}{\partial y}(\tau, 1) =  h^{\prime}(\tau)u(\tau, 1) \;\;\; \text{for}\;\; \tau\in(0, T), \\
&\label{aa20} h^{\prime}(\tau) =A_0\left(u(\tau, 1)-\frac{\sigma(h(\tau))}{m_{0}}\right)\;\;\; \text{for}\;\; \tau\in(0, T)\\
\label{aa21} &u(0, y)  = u_0(y) \;\;\; \text{for}\;\; y\in[0, 1], \\
\label{aa22}&h(0) = 1.
\end{align} 
We refer to the system \eqref{aa17}--\eqref{aa22} posed in the cylindrical domain  $Q(T)$  as problem $(P)$. 
\begin{definition}
	\label{D1}
	{\rm (Weak Solution to ($P$))}. We call the couple $(u, h)$ a weak solution to problem {\rm ($P$)} on $S_T:=(0, T)$ if and only if  
	\begin{align*}  &h\in W^{1, \infty}(S_{T}) \;\;\text{with}\;\;h_0 < h(T) \leq L,\\
	&u \in W^{1,2}(Q(T)) \cap L^{\infty}(S_{T}, H^1(0, 1)) \cap L^2(S_{T}, H^2(0, 1)),
	\end{align*} 
	such that for all $\tau \in S_{T}$ the following relations hold 
	\begin{align}
	\nonumber\displaystyle \left(\frac{\partial u}{\partial \tau}, \varphi \right)&-  \frac{h^{\prime}(\tau)}{h(\tau)}\left( y\frac{\partial u}{\partial y},  \varphi \right) +  \frac{1}{(h(\tau))^2}\left(\frac{\partial{u}}{\partial{y}}, \frac{\partial \varphi }{\partial y}  \right)\\
	\label{a17}& - \frac{1}{h(\tau)} {\rm Bi}\left(\frac{b(\tau)}{m_{0}}- {\rm H}u(\tau, 0)\right)\varphi(0) +\frac{h^{\prime}(\tau)}{h(\tau)}u(\tau, 1)\varphi (1) = 0 \;\;\text{for all} \; \varphi \in H^1(0,1),\\
	& \label{a18} h^{\prime}(\tau) =A_0\left(u(\tau, 1)-\frac{\sigma(h(\tau))}{m_{0}}\right),\\
	\label{u}&u(0, y) = u_0(y)\;\; \text{for}\;\; y \in[0, 1],  \\
	\label{h}&h(0) = 1.
	\end{align}
\end{definition}
\begin{theorem}
	If {\rm \ref{A1}--\ref{A5}} hold, then the problem $(P)$ has a unique solution $(u, h)$ on $S_T$ in the sense of Definition \ref{D1}.
\end{theorem}
\begin{proof} We refer the reader to Theorem 3.3 and Theorem 3.4 in \cite{kumazaki2020global} for a way to ensure the global solvability of the problem and continuous dependence estimates of the solution  with respect to the initial data..
\end{proof}

We discretize the fixed domain $\Omega = (0, 1)$ as follows.
Let $N \in \mathbb{N}$ be 
given.
We set $0 =y_0 < y_1 < \cdots < y_{N-1} = 1$ as discretization points in the interval $[0, 1]$. We set $k_i := y_{i+1} - y_{i}$  for $i \in \{0,1,\cdots, N-2\}$ and  $k := \displaystyle \max \{k_i: \ i \in \{0,1,\cdots, N-2\}\}$. We introduce the space 
\begin{align} \label{finite}
V_k := \{ \nu \in C[0, 1]: \nu |_{[y_j, y_{j+1}]}  \in \mathbb{P}_1\}, 
\end{align} 
as a finite dimesnional subspace of $H^1(0,1)$. Here $\mathbb{P}_1$ represents the set of 
polynomials of degree one.
We define the interpolation operator $I_k:C[0, 1] \rightarrow V_k$ by
\begin{align*}
(I_k u)(y):= \sum_{i=0}^{N-1} u(y_i, t)\phi_i(y),
\end{align*}
where $\{\phi_i\}_{i=0}^{N-1}$ are a set of basis functions for the space $V_k$.
 Here the function $I_k u$ is called the Lagrange interpolant of $u$ of degree 1; for more details see e.g. \cite{larsson2008partial, asadzadeh2020introduction}. 
\begin{lemma} \label{lemma1}
	Take $\theta \in [\frac{1}{2}, 1)$ and $\psi\in H^2(0,1)$.  Then there exist  strictly positive constants $\gamma_1,\; \gamma_2$ and $\gamma_3$ such that the Lagrange interpolant $I_k \psi$ of $\psi$ satisfies the following estimates:
	\begin{enumerate}[label = {\rm (\roman*)}]
		\item  	\label{l1} $\|\psi - I_k\psi \|\leq \gamma_1 k^2 \|\psi \|_{H^2(0,1)}$, 
		\item  \label{l2} $\left\Vert\displaystyle\frac{\partial}{\partial y}  (\psi- I_k\psi) \right\Vert\leq \gamma_2 k \|\psi \|_{H^2(0,1)}$ ,
		\item 	\label{l3}	$|\psi(0) - I_k\psi(0)|\leq  \hat{c} \left(  \gamma_1 k^2 + \gamma_3 k^{1+\theta} \right) \|\psi \|_{H^2(0,1)}$, 
		\item \label{l4}	$|\psi(1) - I_k \psi(1)|\leq \hat{c} \left(  \gamma_1 k^2 + \gamma_3 k^{1+\theta} \right) \|\psi \|_{H^2(0,1)}$. 
	\end{enumerate}
\end{lemma}
\begin{proof}
 The inequalities \ref{l1} and \ref{l2} are standard results,  see for instance p.61 in \cite{larsson2008partial} and p. 90 in \cite{asadzadeh2020introduction} for details on their proof. The proof of \ref{l3} and \ref{l4} follows from using the interpolation inequality, \ref{l1} and \ref{l2}.  We refer the reader to Lemma 2.1 in \cite{nepal2021error} for details on their proof..
\end{proof}

The continuous in time finite element approximation $u_k$ and $h_k$ of $u$ and $h$ is now represented by the mappings
$u_k: [0, T]\rightarrow V_k$ and  $h_k: [0, T]\rightarrow \mathbb{R}_+$. Our concept of solution to the semi-discrete problem is defined next.

\begin{definition}
	\label{D2}
	 We call the couple $(u_k, h_k)$ a weak solution to the semi-discrete formulation if and only if there is a $S_{{T}}:= (0, {T})$  (for some ${T}>0$) such that 
	\begin{align*}  & h_k\in W^{1, \infty}(S_{{T}}) \;\; \text{with}\;\; h_0 < h_k({T}) \leq L\\
	 &  u_k \in H^1(S_{{T}}, V_k)\cap L^{2}(S_{{T}}, H^1(0, 1)) \cap L^{\infty}(S_{{T}}, L^2(0, 1))
	\end{align*} 
	and for all $\tau \in S_{{T}}$ it holds
\begin{align}
	\nonumber\displaystyle \left( \frac{\partial u_k}{\partial \tau},  \varphi_k \right)&-  \frac{h_k^{\prime}(\tau)}{h_k(\tau)}\left( y\frac{\partial u_k}{\partial y}, \varphi_k \right) +  \frac{1}{(h_k(\tau))^2}\left(\frac{\partial{u_k}}{\partial{y}}, \frac{\partial\varphi_k }{\partial y} \right)\\
	\label{a19}&- \frac{1}{h_k(\tau)} {\rm Bi}\left(\frac{b(\tau)}{m_{0}}- {\rm H}u_k(\tau, 0)\right) \varphi_k(0) +\frac{h_k^{\prime}(\tau)}{h_k(\tau)}u_k(\tau, 1)\varphi_k (1) = 0 \;\; \text{for all}\;\; \varphi_k \in V_k, \\
	\label{a20} & h_k^{\prime}(\tau) = A_0\left(u_k(\tau, 1)-\frac{\sigma(h_k(\tau))}{m_0}\right),\\
	\label{a21}& u_k(0) = u_{0,k}(y) \;\; {\rm for}\;\; y \in[0, 1],\\
	\label{a22}& h_k(0) = 1.
\end{align}
\end{definition}
\begin{theorem}
	Assume {\rm \ref{A1}--\ref{A5}} hold.
	Then  it exists a unique solution 
	\begin{align*}
	  (u_k, h_k) \in H^1(S_{\hat{T}}, V_k)\cap L^{2}(S_{\hat{T}}, H^1(0, 1)) \cap L^{\infty}(S_{\hat{T}}, L^2(0, 1)) \times W^{1, \infty}(S_{\hat{T}})
	\end{align*}
on a time $S_{\hat{T}}:=(0, \hat{T})$ for  $\hat{T}\in (0, T]$	in the sense of Definition \ref{D2}. Furthermore, there exists a constant $\tilde{c}>0$ (independent of $k$) such that
\begin{align}
\label{energy}	 \max_{0\leq \tau\leq \hat{T}} \| u_k\|_{L^2(0, 1)} ^2 + \int_0^{\hat{T}} \left\Vert \frac{\partial u_k}{\partial y}\right\Vert^2_{L^2(0,1)}d \tau \leq \tilde{c}.
\end{align}
\end{theorem}
\begin{proof}
We refer the reader to Theorem 4.2 in our previous work \cite{nepal2021error} for details of the proof.
\end{proof}
We anticipate in Proposition \ref{higher} a regularity result that is going to be useful when proving convergence rates for the proposed fully-discrete approximation. Mind though that the convergence of our scheme holds for much less regularity than stated. 

\begin{preposition}\label{higher}
(Higher regularity). Assume that {\rm \ref{A1}-\ref{A5}} hold together with $b\in W^{2,2}(0, T)$ and $u_0 \in H^2(0, s_0)$. Then  
	\begin{align*}  &p\in W^{1, \infty}(S_{T}) \\
	&w \in W^{1,2}(Q(T)) \cap L^{\infty}(Q(T)) \cap L^2(S_{T}, H^2(0, 1)),
	\end{align*} 
	where $p:= h^{\prime}$, $w:= \partial_\tau u$, and $(u, h)$ is the weak solution in the sense of Definition \ref{D1}. Additionally, if $b \in W^{2+\theta, 2},\; \theta\in(0, 1]$, then  $\partial_\tau w \in L^{\infty}(Q(T))$.
\end{preposition}
\begin{proof}
One differentiates with respect to time \eqref{aa17}-\eqref{aa22} and proves for the resulting system of equations a statement as Theorem 3.1 in \cite{kumazaki2020global}. Note that the embedding $H^{1+\theta}(Q(T)) \subset C(\overline{Q(T)})$ implies 
$\partial_{\tau\tau} u \in L^{\infty}((Q(T)))$.
\end{proof}

\section{Fully discrete error analysis}\label{fully}

In this section, we present firstly a fully discrete numerical scheme of the problem \eqref{aa17}-\eqref{aa22} and then perform  the error analysis.  Let $n, M\in \mathbb{N}$. Concerning the discretization in time, we decompose the interval $(0, T]$ in to $M$ subintervals.  Let $\Delta \tau := T/M$ be a step size of the time variable. Define $\tau^n := n \Delta \tau$, for $n \in \{0, 1,2, \cdots, M\}$.  At any time level $\tau^n$, we denote $u(\tau^n, x)$ for $x\in[0, 1]$ and $h(\tau^n)$ by $u^n$ and $h^n$ respectively. Furthermore, we denote the fully discrete approximation of $u^n$ and $h^n$  by  $U^n$ and $W^n$  respectively. We use the following notations for the time derivative:
\begin{align*}
\Delta_\tau U^n := \frac{U^{n+1}- U^{n}}{\Delta \tau} \;\;\; \text{and}\;\;\;\Delta_\tau W^n := \frac{W^{n+1}- W^{n}}{\Delta \tau}. 
\end{align*} 
Using these notations, the fully discrete problem is formulated as follows: 
Find the pair $(U^{n+1},W^{n+1})\in V_k\times\mathbb{R}^+$  such that   the following system holds for all $n \in \{0, \cdots, M-1 \}$: 
\begin{align}
\nonumber\displaystyle \left( \Delta_\tau U^n,  \varphi_k \right) & -  \frac{\Delta_\tau W^n}{W^{n+1}}
 \left( y\frac{\partial U^{n+1}}{\partial y}, \varphi_k \right) + \frac{1}{(W^{n+1})^2}\left(\frac{\partial{U^{n+1}}}{\partial{y}}, \frac{\partial\varphi_k }{\partial y} \right)  \\
 &\label{b1} - \frac{1}{W^{n+1}} {\rm Bi}\left(\frac{b(\tau)}{m_{0}} - {\rm H}U^n(0)\right) \varphi_k(0) +\frac{\Delta_\tau W^n}{W^{n+1}}U^n(1)\varphi_k (1) = 0\;\; \text{for all}\;\; \varphi_k \in V_k, \\
& \label{b2}\Delta_\tau W^n =  A_0\left(U^n(1) -\frac{\sigma(W^n)}{m_0}\right),\\
& \label{b3}U^0 = U_0(0),\\
& \label{b4}W^0 = 1,
\end{align}
where $U_0(0)$ is an  appropriate approximation of the initial condition $u_0$ in $V_k$. 

\begin{lemma} \label{Lemma1} Assume that {\rm\ref{A1}-\ref{A5}} hold. Then
there exist $T^{\prime}\in (0 , T]$  and positive constants $\tilde{K}, \underbar{u},\; \bar{u}$ and $\bar{W}$ such that for all $ n\in \mathbb{N}$ the following inequalities hold true:
\begin{enumerate}[label = {\rm (\roman*)}]
\item \label{third} $W^0 \leq W^n \leq \bar{W}$,
\item  \label{fourth} $0<\underbar{u}\leq U^n \leq \bar{u}$,
\item \label{second}$0 \leq |\Delta_\tau W^n| \leq \tilde{K}$.
\end{enumerate}

\end{lemma}
\begin{proof}
The property (i) is built in the concept of solution cf. Definition \ref{D2}, while (iii) is a direct consequence of (i) and (ii). The fact that the solution to our problem satisfies a weak maximum principle makes us confident that (ii) holds as well. One way to prove such statement directly would be  to ensure that a discrete maximum principle holds in our situation; we refer for instance to the working techniques used in \cite{farago2012discrete}.
\end{proof}

\begin{preposition} \label{propo}
Assume that {\rm\ref{A1}-\ref{A5}} hold. Then for sufficiently small $\Delta \tau$   the solution $(U^n, W^n)$ of \eqref{b1}-\eqref{b4} satisfies the following estimates:
\begin{enumerate}[label = {\rm (\roman*)}]
\item  \label{Energy}   $\displaystyle\max_{1\leq n \leq M} \|U^n\|_{L^2(0, 1)} ^2 + \Delta \tau \sum_{j=1}^M \left\Vert \frac{\partial U^j}{\partial y}\right\Vert^2_{L^2(0,1)}d \tau \leq K^*$.\\
\item \label{first} $  |W^n| \leq K_1(K^*, A_0, W^0) \;\;\text{for}\;\;  n \in \{1,2, \cdots, M\}$. 
\end{enumerate}
The constants $K^*$ and $K_1(K^*, A_0, W_0)$ depend on data, parameters, as well as on the quality of the approximation of the term $\|U^0\|^2  + \left\lVert \displaystyle \frac{\partial U^0 }{\partial y}\right\rVert^2$.
\end{preposition}
\begin{proof}
Taking $\varphi_k = U^{n+1}$ in \eqref{b1} yields
\begin{align}
  \nonumber\displaystyle \left( \Delta_\tau U^n,  U^{n+1} \right) & -  \frac{\Delta_\tau W^n}{W^{n+1}}
 \left( y\frac{\partial U^{n+1}}{\partial y}, U^{n+1} \right) + \frac{1}{(W^{n+1})^2}\left(\frac{\partial{U^{n+1}}}{\partial{y}}, \frac{\partial U^{n+1} }{\partial y} \right)\\
 &\label{3b11} - \frac{1}{W^{n+1}} {\rm Bi}\left(\frac{b(\tau)}{m_{0}} - {\rm H}U^n(0)\right) U^{n+1}(0) +\frac{\Delta_\tau W^n}{W^{n+1}}U^n(1)U^{n+1}(1) =0.
\end{align}
By using $2(b-a, b) = b^2 - a^2 + (b-a)^2$ and Cauchy–Schwarz's inequality, we get
\begin{align*}
   \nonumber  \frac{1}{2\Delta \tau} (W^{n+1})^2&\left( \| U^{n+1}\|^2 - \| U^n\|^2\right)  + \left\lVert \nonumber \frac{\partial U^{n+1}}{\partial y} \right\lVert^2 \\
    \nonumber  \leq& |W^{n+1}| |\Delta_{\tau}W^n| \left\lVert \frac{\partial U^{n+1}}{\partial y} \right\lVert \left\lVert U^{n+1} \right\lVert
    + {\rm Bi}|W^{n+1}| \frac{b^*}{m_0} |U^{n+1}(0)| \\
    &+ {\rm Bi H}|W^{n+1}| |U^n(0)||U^{n+1}(0)| + |W^{n+1}| |\Delta_{\tau}W^n|   |U^n(1)||U^{n+1}(1)|\\
     \leq& c_{\xi} \bar{W}^2\tilde{K}^2   \| U^{n_+1}\|^2 + \xi \left\lVert \frac{\partial U^{n+1}}{\partial y}\right\rVert^2
    + {\rm Bi} \bar{W} \frac{b^*}{m_0} \| U^{n+1} \|_{H^1} \\
    &+ {\rm Bi H}\bar{W} |U^n(0)||U^{n+1}(0)| + \bar{W} \tilde{K}  |U^n(1)||U^{n+1}(1)|\\
    \leq& c_{\xi} \bar{W}^2\tilde{K}^2   \| U^{n_+1}\|^2 + \xi \left\lVert \frac{\partial U^{n+1}}{\partial y}\right\rVert ^2
    + c_{\xi}{\rm Bi^2} \bar{W}^2 \frac{(b^*)^2}{m_0^2} + \xi \| U^{n+1} \|^2 \\ 
    & +  \xi \left\lVert \frac{\partial U^{n+1}}{\partial y}\right\rVert^2 + ({\rm Bi\; H}+\tilde{K})\bar{W}\hat{c}  \left\lVert U^n \right\lVert^{\theta}  \left\lVert  U^n \right\lVert_{H^1(0, 1)}^{1-\theta}  \left\lVert U^{n+1} \right\lVert^{\theta}  \left\lVert  U^{n+1} \right\lVert_{H^1(0, 1)}^{1-\theta}\\
    \leq & c_{\xi} \bar{W}^2\tilde{K}^2   \| U^{n_+1}\|^2 + \xi \left\lVert \frac{\partial U^{n+1}}{\partial y}\right\rVert ^2
    + c_{\xi}{\rm Bi^2} \bar{W}^2 \frac{(b^*)^2}{m_0^2} + \xi \| U^{n+1} \|^2 \\ 
    & +  \xi \left\lVert \frac{\partial U^{n+1}}{\partial y}\right\rVert^2 + ({\rm Bi H}+ \tilde{K})\bar{W} 
 \frac{\hat{c}}{2}  \left(\left\lVert U^n \right\lVert \left\lVert  U^n  \right\lVert_{H^1(0, 1)} + \left\lVert U^{n+1} \right\lVert  \left\lVert  U^{n+1} \right\lVert_{H^1(0, 1)}\right)\\
 \leq& c_{\xi} \bar{W}^2\tilde{K}^2   \| U^{n_+1}\|^2 + \xi \left\lVert \frac{\partial U^{n+1}}{\partial y}\right\rVert^2
    + c_{\xi}{\rm Bi^2} \bar{W}^2 \frac{(b^*)^2}{m_0^2} + \xi \| U^{n+1} \|^2 +  \xi \left\lVert \frac{\partial U^{n+1}}{\partial y}\right\rVert^2\\ 
    &  + ({\rm Bi H}\bar{W}  +  \tilde{K})\bar{W} \frac{\hat{c}}{2}  \left((\zeta+ c_{\zeta}) \left\lVert U^n \right\lVert^2 + \zeta  \left\lVert \frac{\partial U^n}{\partial y}  \right\lVert^2 + (\xi+ c_{\xi}) \left\lVert U^{n+1} \right\lVert^2 + \xi  \left\lVert \frac{\partial U^{n+1}}{\partial y} \right\lVert^2 \right)\\
    \leq& C_1(\zeta, c_{\zeta})\| U^n\|^2 + C_2 (\zeta)\left\lVert \frac{\partial U^n}{\partial y}\right\lVert^2 + C_3(\xi, c_{\xi})\| U^{n+1}\|^2 + C_4 (\xi) \left\lVert \frac{\partial U^{n+1}}{\partial y}\right\lVert^2 + C_5(c_{\xi}).
\end{align*}
Adding $(1-C_4(\xi) - C_2(\zeta))\left\lVert  \frac{\partial U^{n}}{\partial y} \right\lVert^2$ on both sides and multiplying by $2\Delta t$ gives
\begin{align}
    \nonumber    (1- 2\Delta \tau C_1) &\| U^{n+1}\|^2  + 2\Delta \tau (1- C_4(\xi))\left\lVert  \frac{\partial U^{n+1}}{\partial y} \right\lVert^2 + 2\Delta \tau(1- C_4(\xi) - C_2(\zeta))\left\lVert  \frac{\partial U^{n}}{\partial y} \right\lVert^2 \\
   \label{3ca} & \leq (1+ 2\Delta \tau C_1(\xi, c_{\xi}))\| U^n\|^2 + 2\Delta \tau (1- C_4(\xi))\left\lVert  \frac{\partial U^{n}}{\partial y} \right\lVert^2 + 2\Delta \tau C_5,
\end{align}
 where $C_1:= \max\{C_1(\xi, c_{\xi}), C_3(\zeta, c_{\zeta})\}$.
 With the notation: $c_n:= \| U^{n} \|^2$ and $d_n:=\left\lVert  \frac{\partial U^{n}}{\partial y} \right\lVert^2$,
the inequality \eqref{3ca} can be written as follows:
\begin{align}\label{sum}
  c_{n+1} + a_1  d_{n+1} + a_2 d_{n} &\leq a_3 c_{n} + a_1 d_{n} + a_4,
\end{align}
where
\begin{align*}
a_1 := \frac{2\Delta \tau (1-C_4)}{1- 2\Delta \tau C_1},\; a_2 := \frac{2\Delta \tau (1-C_4 -C_2)}{1- 2\Delta \tau C_1},\;  a_3 := \frac{1+ 2 \Delta \tau C_1}{1- 2\Delta \tau C_1},\; a_4 := \frac{2 \Delta \tau C_5}{1 - 2\Delta \tau C_1}.
\end{align*}
Using $a_3>1$ in \eqref{sum}, we get the following inequalities:
\begin{align}
\label{3c3}c_{n+1} + a_1  d_{n+1} + a_2 d_{n} &\leq a_3 ( c_{n} + a_1 d_{n} )+ a_4,\\
\label{3c4}a_3(c_n + a_1 d_n) + a_3 a_2 d_{n-1} &\leq a_3^2  (c_{n-1} + a_1 d_{n-1} )+ a_3 a_4,\\
\label{3c5}a_3^2(c_{n-1} + a_1 d_{n-1}) + a_3^2 a_2 d_{n-2} &\leq a_3^3  (c_{n-2} + a_1 d_{n-2} )+ a_3^2 a_4,\\
\nonumber&\vdots\\
\label{3c6}a_3^n(c_{1} + a_1 d_{1}) + a_3^n a_2 d_{0} &\leq a_3^{n+1}  (c_{0} + a_1 d_{0} )+ a_3^n a_4.
\end{align}
Adding $\eqref{3c3}-\eqref{3c6}$, we obtain
\begin{align*}
c_{n+1} + a_1  d_{n+1} + a_2 [d_{n} + a_3 d_{n-1} + a_3^2 d_{n-2} + \cdots +  a_3^n d_0] \leq a_3^{n+1}  (c_{0} + a_1 d_{0} )+ a_4  [1 + a_3 + a_3^2+ \cdots + a_3^n ].
\end{align*}
For sufficiently small $\Delta \tau$ and  $n\in \{0, 1, \cdots, M-1 \}$,  we can write
\begin{align}
 \label{bb9}\| U^{n+1} \|^2   +  a_2  \sum_{j=1}^{n+1}  \left\lVert \frac{\partial U^{j} }{\partial y}  \right\lVert^2 \leq a_3^{n+1}  \left(\|U^0\|^2  + a_1  \left\lVert \frac{\partial U^0 }{\partial y}  \right\lVert^2 \right)+ a_4  \sum_{j=0}^{n}  a_3^j .
\end{align}
It follows from \eqref{bb9}  
\begin{align*}
 \max_{1\leq n \leq M }\| U^{n} \|^2   +  a_2 \sum_{j=1}^{M}  \left\lVert \frac{\partial U^{j} }{\partial y}  \right\lVert^2 \leq a_3^{M}  \left(\|U^0\|^2  + a_1  \left\lVert \frac{\partial U^0 }{\partial y}  \right\lVert^2 \right)+ a_4  \sum_{j=0}^{M-1}  a_3^j.
\end{align*}
This completes the proof of \ref{Energy}. 
To prove \ref{first}, we use  \eqref{b2} to get
\begin{align*}
&W^{n+1} = W^n + \Delta \tau  A_0 \left(U^n(1) - \frac{\sigma(W^n)}{m_0} \right).
\end{align*}
Furthermore, we get 
\begin{align*}
|W^{n+1}| &\leq |W^n| + \Delta \tau  A_0 \left( |U^n(1)| + \frac{|\sigma(W^n)|}{m_0} \right)\\
& \leq |W^n| + \Delta \tau  A_0 \left( \hat{c} \|U^n\|^{\theta} \left\Vert U^n\right\Vert_{H^1(0,1)}^{1 - \theta}  + \frac{|\sigma(W^n)|}{m_0}   \right)\\
& \leq |W^n| + K(K^*, A_0).
\end{align*}
Repeating the same procedure to bound $|W^n|$  gives 
\begin{align}
| W^n| \leq K(K^*, A_0) + |W^0| \;\;\text{for}\;\;  n\in \{ 1,2, \cdots, M\}.
\end{align}
\end{proof}
\begin{remark}
The statement \ref{Energy} in Proposition \ref{propo} can be seen as the discrete version of \eqref{energy}. It provides  information regarding the stability of our fully discrete scheme. It turns out that the result holds for a sufficiently small $\Delta \tau$ for a fixed $n$.  The restriction on time step size indicates the our discretization scheme is possibly conditionally stable.  
\end{remark}
We now discuss the existence of solution to the fully discrete problem \eqref{b1}-\eqref{b4}.
\begin{theorem}
Assume  {\rm\ref{A1}-\ref{A5}} hold. Then there exists a solution for the  fully discrete problem \eqref{b1}-\eqref{b4}. 
\end{theorem}
\begin{proof} Starting from $W^0$ and $U^0$,  the existence of $W^{1}$ comes from the existence of $\sigma(W^0)$. For the existence of the solution for the concentration profile $U^{n+1}$, we use the application of the Brouwer fixed point theorem. We refer the reader to p. 206 in \cite{kesavan1989topics}, Lemma 4.2 in \cite{alam2021existence}, as well as  to Lemma 1.4 in \cite{temam2001navier} for more functional analytic details. 
The space $V_k$ is a separable as a subspace of $H^1(0,1)$. Let $\{w_1, w_2, \cdots, \}$ be an orthogonal basis of $H^1(0, 1)$ and an orthonormal basis for $L^2(0,1)$.
For each fixed integer $m\geq 1$, we define an approximation solution $U^{n+1}$ of \eqref{b1} by
\begin{align}
 \label{lin} U^{n+1}:= \sum_{i=1}^{m} \zeta_i w_i.
\end{align}
We define $F:\mathbb{R}^m \rightarrow \mathbb{R}^m $ by

\begin{align}
(F(\zeta))_i :=  \nonumber\displaystyle \left( \Delta_\tau U^n,  w_i \right) & -  \frac{\Delta_\tau W^n}{W^{n+1}}
 \left( y\frac{\partial U^{n+1}}{\partial y}, w_i \right) + \frac{1}{(W^{n+1})^2}\left(\frac{\partial{U^{n+1}}}{\partial{y}}, \frac{\partial w_i }{\partial y} \right)  \\
 &\label{b11} - \frac{1}{W^{n+1}} {\rm Bi}\left(\frac{b(\tau)}{m_{0}} - {\rm H}U^n(0)\right) w_i(0) +\frac{\Delta_\tau W^n}{W^{n+1}}U^n(1)w_i(1),
\end{align}
where $U^{n+1}$ is defined in \eqref{lin}. Thus, \eqref{b11} will have a solution if there exists a $\zeta$ such that $F(\zeta) = 0$. 
\begin{align}
  \nonumber \left( F(\zeta), \zeta\right) &=  \sum_{i=1}^{m} (F(\zeta))_i \zeta_i\\
  &= \nonumber\displaystyle \left( \Delta_\tau U^n,  U^{n+1} \right)  -  \frac{\Delta_\tau W^n}{W^{n+1}}
 \left( y\frac{\partial U^{n+1}}{\partial y}, U^{n+1} \right) + \frac{1}{(W^{n+1})^2}\left(\frac{\partial{U^{n+1}}}{\partial{y}}, \frac{\partial U^{n+1} }{\partial y} \right)\\
 &\label{1b11} - \frac{1}{W^{n+1}} {\rm Bi}\left(\frac{b(\tau)}{m_{0}} - {\rm H}U^n(0)\right) U^{n+1}(0) +\frac{\Delta_\tau W^n}{W^{n+1}}U^n(1)U^{n+1}(1).
\end{align}
By using $2(b-a, b) = b^2 - a^2 + (b-a)^2$ for the first term, integration by part formula for the second term on right hand side of \eqref{1b11}, we get,
\begin{align}
\nonumber    \left( F(\zeta), \zeta\right) \geq & \frac{1}{2 \Delta \tau} \left( \| U^{n+1}\|^2 - \| U^n\|^2\right) - \frac{\Delta_\tau W^n}{W^{n+1}} \left( (U^{n+1})^2(1) - (U^{n+1})^2(0) \right)\\
  \nonumber  & + \frac{1}{(W^{n+1})^2}\left\lVert \frac{\partial  U^{n+1} }{\partial y}\right\lVert^2 -  \frac{1}{W^{n+1}} {\rm Bi}\frac{b^*}{m_{0}} U^{n+1}(0)\\
  \nonumber & +  \frac{1}{W^{n+1}}{\rm H}\,{\rm Bi}\, U^n(0) U^{n+1}(0) + \frac{\Delta_\tau W^n}{W^{n+1}} U^n(1) U^{n+1}(1). 
  \end{align}
  Using Lemma \ref{Lemma1}, we obtain
  \begin{align}
  \nonumber \left( F(\zeta), \zeta\right) 
    \geq & \frac{1}{2 \Delta \tau} \left( \| U^{n+1}\|^2 - \| U^n\|^2\right) - \frac{1}{\Delta \tau}\left( 1 - \frac{W^n}{W^{n+1}}\right) \bar{u}^2\\
 \nonumber   &  -  \frac{1}{W^{n+1}} {\rm Bi}\frac{b^*}{m_{0}}\bar{u} + \frac{1}{W^{n+1}}{\rm H}\,{\rm Bi}\, (\underbar{u})^2 + \frac{\Delta_\tau W^n}{W^{n+1}}(\underbar{u})^2\\
\nonumber  \geq & \frac{1}{2 \Delta \tau} \left( \| U^{n+1}\|^2 - \| U^n\|^2\right) - \frac{1}{\Delta \tau}\left( 1 - \frac{W^n}{W^{n+1}}\right) \bar{u}^2 -  \frac{1}{W^{n+1}} {\rm Bi}\frac{b^*}{m_{0}}\bar{u}\\
 \nonumber \geq & \frac{1}{2\Delta\tau} \| U^{n+1}\|^2 - \left\{\frac{1}{2\Delta\tau} \left( \| U^{n}\|^2 + \left( 1 - \frac{W^n}{W^{n+1}}\right)(\bar{u})^2\right) + \frac{1}{W^{n+1}} {\rm Bi} \frac{b^*}{m_0}\bar{u}\right\}. 
\end{align}
We choose $ \| U^{n+1}\|^2 = R^2$ large enough such that $\left( F(\zeta), \zeta\right)>0$. This completes the proof. 
\end{proof}
\begin{theorem}
Assume that {\rm\ref{A1}--\ref{A5}} hold. Then the problem \eqref{b1}-\eqref{b4} admits an unique solution.
\end{theorem}
\begin{proof}
Assume  by contradiction that $(U^{n+1}, W^{n+1})$ and $(\bar{U}^{n+1}, \bar{W}^{n+1})$ are two solutions satisfying \eqref{b1}-\eqref{b4} with the same initial data.  Then we get the following holds for all $\varphi_k \in V_k$:
\begin{align}
\nonumber\displaystyle \left( \Delta_\tau U^n,  \varphi_k \right) & -  \frac{\Delta_\tau W^n}{W^{n+1}}
 \left( y\frac{\partial U^{n+1}}{\partial y}, \varphi_k \right) + \frac{1}{(W^{n+1})^2}\left(\frac{\partial{U^{n+1}}}{\partial{y}}, \frac{\partial\varphi_k }{\partial y} \right)  \\
 &\label{1b1} - \frac{1}{W^{n+1}} {\rm Bi}\left(\frac{b(\tau)}{m_{0}} - {\rm H}U^n(0)\right) \varphi_k(0) +\frac{\Delta_\tau W^n}{W^{n+1}}U^n(1)\varphi_k (1) = 0, \\
 \nonumber\displaystyle \left( \Delta_\tau \bar{U}^n,  \varphi_k \right) & -  \frac{\Delta_\tau \bar{W}^n}{\bar{W}^{n+1}}
 \left( y\frac{\partial \bar{U}^{n+1}}{\partial y}, \varphi_k \right) + \frac{1}{(\bar{W}^{n+1})^2}\left(\frac{\partial{\bar{U}^{n+1}}}{\partial{y}}, \frac{\partial\varphi_k }{\partial y} \right)  \\
 &\label{2b1} - \frac{1}{\bar{W}^{n+1}} {\rm Bi}\left(\frac{b(\tau)}{m_{0}} - {\rm H}\bar{U}^n(0)\right) \varphi_k(0) +\frac{\Delta_\tau \bar{W}^n}{\bar{W}^{n+1}}\bar{U}^n(1)\varphi_k (1) = 0.
\end{align}
Furthermore, we have
\begin{align}
 & \label{1b2}\Delta_\tau W^n =  A_0\left(U^n(1) -\frac{\sigma(W^n)}{m_0}\right),\\
& \label{2b2}\Delta_\tau \bar{W}^n =  A_0\left(\bar{U}^n(1) -\frac{\sigma(\bar{W}^n)}{m_0}\right).
\end{align}
We will show that these two solutions  must coincide.
We use  the method of induction to prove $U^{n+1} = \bar{U}^{n+1}$ and $W^{n+1} = \bar{W}^{n+1}$ for all $n\in \{-1, 0, 1, \cdot, \}$. Obviously, it holds for   $n = -1$.  Assume that the statement holds for an arbitrarily fixed $n\in \mathbb{N}$, i.e, $U^{n} = \bar{U}^{n} $.  It remains to show  $U^{n+1} = \bar{U}^{n+1}$ and $W^{n+1} = \bar{W}^{n+1}$. 
Subtracting \eqref{2b2} from \eqref{1b2}, and using the induction hypothesis gives  $W^{n+1} = \bar{W}^{n+1}$ i.e., $\Delta_\tau W^n = \Delta_\tau \bar{W}^n$. Indeed,
\begin{align*}
  \nonumber \frac{1}{\Delta \tau} (W^{n+1} - \bar{W}^{n+1} - (W^{n} -\bar{W}^{n}))& = \frac{A_0}{m_0} \left( \sigma(W^n)-\sigma(\bar{W}^n)\right)\\
 |W^{n+1} - \bar{W}^{n+1}| &\leq \left(1+ \frac{\Delta \tau \mathcal{L}A_0}{m_0}\right) |W^{n} -\bar{W}^{n}|.
\end{align*}
By repeating the same process, it yields
\begin{align*}
  |W^{n+1} - \bar{W}^{n+1}| \leq  \left(1+ \frac{\Delta \tau \mathcal{L}A_0}{m_0}\right)^2 |W^{n-1} -\bar{W}^{n-1}|\leq \cdots \leq \left(1+ \frac{\Delta \tau \mathcal{L}A_0}{m_0}\right)^{(n+1)} |W^{0} -\bar{W}^{0}|. 
\end{align*}
It now remains to show $U^{n+1} = \bar{U}^{n+1}$. We subtract \eqref{2b1} from \eqref{1b1} and use the induction hypothesis to obtain
\begin{align}
  \label{unique} \frac{1}{\Delta \tau} \left(U^{n+1} - \bar{U}^{n+1}, \varphi_k\right) - \frac{\Delta_\tau W^n}{W^{n+1}}  \left( y \frac{\partial}{\partial y}(U^{n+1} - \bar{U}^{n+1}), \varphi_k \right) + \frac{1}{(W^{n+1})^2} \left( \frac{\partial }{\partial y}(U^{n+1} - \bar{U}^{n+1}), \frac{\partial \varphi_k}{\partial y}\right) = 0.
\end{align}
Choosing $\varphi_k = U^{n+1} - \bar{U}^{n+1}$
in \eqref{unique}, it yields

\begin{align}
 \nonumber  \frac{1}{\Delta \tau} \left(U^{n+1} - \bar{U}^{n+1}, U^{n+1} - \bar{U}^{n+1}\right) &- \frac{\Delta_\tau W^n}{W^{n+1}}  \left( y \frac{\partial}{\partial y}(U^{n+1} - \bar{U}^{n+1}), U^{n+1} - \bar{U}^{n+1} \right) \\
  \nonumber &+ \frac{1}{(W^{n+1})^2} \left( \frac{\partial }{\partial y}(U^{n+1} - \bar{U}^{n+1}), \frac{\partial }{\partial y}(U^{n+1} - \bar{U}^{n+1})\right) = 0\\
 \nonumber \frac{1}{\Delta \tau} \left\|U^{n+1} - \bar{U}^{n+1}\right\|^2 &+ \frac{1}{(W^{n+1})^2} \left\| \frac{\partial }{\partial y}(U^{n+1} - \bar{U}^{n+1})\right\|^2\\
  \nonumber & \leq \frac{\Delta_\tau W^n}{W^{n+1}}  \left\| \frac{\partial}{\partial y}(U^{n+1} - \bar{U}^{n+1})\right\| \left\| U^{n+1} - \bar{U}^{n+1} \right\|. 
 \end{align}
 Using Lemma \ref{Lemma1} and Young's inequality, we get 
 \begin{align}
   \nonumber  \frac{1}{\Delta \tau} \left\|U^{n+1} - \bar{U}^{n+1}\right\|^2 &+ \frac{1}{\bar{W}^2} \left\| \frac{\partial }{\partial y}(U^{n+1} - \bar{U}^{n+1})\right\|^2 \\
  \nonumber &\leq \frac{\tilde{K}}{W^0}  \left\| \frac{\partial}{\partial y}(U^{n+1} - \bar{U}^{n+1})\right\| \left\| U^{n+1} - \bar{U}^{n+1} \right\| \\
  \nonumber \frac{1}{\Delta \tau} \left\|U^{n+1} - \bar{U}^{n+1}\right\|^2 &+ \frac{1}{\bar{W}^2} \left\| \frac{\partial }{\partial y}(U^{n+1} - \bar{U}^{n+1})\right\|^2 \\
 \nonumber &\leq \xi \frac{\tilde{K}^2}{(W^0)^2}  \left\| \frac{\partial}{\partial y}(U^{n+1} - \bar{U}^{n+1})\right\|^2  + c_{\xi}\left\| U^{n+1} - \bar{U}^{n+1} \right\|^2 \\
 \label{unique1}\left(\frac{1}{\Delta \tau} - c_{\xi}\right) \left\|U^{n+1} - \bar{U}^{n+1}\right\|^2 & + \left(\frac{1}{\bar{W}^2} - \xi \frac{\tilde{K}^2}{(W^0)^2} \right)\left\| \frac{\partial }{\partial y}(U^{n+1} - \bar{U}^{n+1})\right\|^2 \leq 0.
\end{align}
Choosing  $\xi \leq (\frac{W^0}{\tilde{K} \bar{W}})^2$  and $\Delta \tau \leq 1/ c_{\xi}$ in \eqref{unique1} gives $\left\|U^{n+1} - \bar{U}^{n+1}\right\|^2\leq 0$. It implies $U^{n+1} = \bar{U}^{n+1}$ a.e. in $(0,1)$. This completes the proof.
\end{proof}

We next analyze error estimates of our fully discrete scheme for the concentration profile and the position of the moving boundary. To estimate the errors $e^n:= U^n - u^n$ and $e_1^n:= W^n - h^n$,  we decompose  $e^n$ into two parts:
\begin{align}
e^n =  U^n - u^n = \psi^n + \rho ^n,
\end{align}
with $ \psi^n:= U^n - I_ku^n$ and  $ \rho ^n:= I_k u^n  - u^n$. Here $I_k u^n$ is a Lagrange interpolation of $u^n$ defined in Lemma \ref{lemma1}.
In the rest of section we derive error estimates for the fully discrete scheme \eqref{b1}-\eqref{b4}. To begin with, we perform the error bound for $\psi^{n+1}$ in the following theorem.
\begin{theorem} \label{theorem1}
	Assume  {\rm\ref{A1}-\ref{A5}} hold together with the hypothesis of Proposition \ref{higher}. Let $(u, h)$  be the corresponding weak solution to problem \eqref{aa17}-\eqref{aa22}  in the sense of Definition \ref{D1}. Let $(U^n, W^n)$ be the solution for the fully discrete formulation \eqref{b1}-\eqref{b4}.  Then there exists a constant $K >0$ such that the following inequality holds for sufficiently small $\Delta \tau$:
\begin{align}
\| \psi^{n+1} \|^2  +   | e_1^{n+1} |^2 + \alpha \Delta \tau \sum_{i=1}^{n+1}  \left\lVert \frac{\partial \psi^{j}}{\partial y}  \right\lVert^2 \leq K \{\Delta \tau ^2 + k^2\}.
\end{align}
\end{theorem}
\begin{proof}  Subtracting  \eqref{a19} from \eqref{b1}, we obtain the following identity:
\begin{align}
\nonumber\displaystyle (W^{n+1})^2 \left( \Delta_\tau U^n,  \varphi_k \right) & - (h^{n+1})^2 \left( \frac{\partial u^{n+1}}{\partial \tau},  \varphi_k \right) - W^{n+1} \Delta_\tau W^n
\left( y\frac{\partial U^{n+1}}{\partial y}, \varphi_k \right) \\
\nonumber &+ h^{n+1}  (h^{n+1})^{\prime} \left( y\frac{\partial u^{n+1}}{\partial y}, \varphi_k \right) 
  + \left(\frac{\partial{U^{n+1}}}{\partial{y}}, \frac{\partial\varphi_k }{\partial y} \right) \\
  \nonumber & - \left( \frac{\partial{u^{n+1}}}{\partial{y}},  \frac{\partial\varphi_k }{\partial y}  \right) -    W^{n+1} {\rm Bi}\left(\frac{b(\tau)}{m_{0}} - {\rm H}U^n(0)\right) \varphi_k(0) \\
  \nonumber  & +  h^{n+1} {\rm Bi}\left(\frac{b(\tau)}{m_{0}}- {\rm H}u^{n+1}(0)\right) \varphi_k(0)  + W^{n+1} \Delta_\tau W^n U^n(1)\varphi_k (1) \\
\label{3b5} & - h^{n+1} (h^{n+1})^{\prime}u^{n+1}(1)\varphi_k (1)  = 0,
\end{align}
which holds for all $\varphi_k \in V_k$.\\
Using $U^n = \psi^n + \rho ^n + u^n$ and 	arranging conveniently the terms in \eqref{3b5} yields 
\begin{align}
\nonumber\displaystyle (W^{n+1})^2 \left( \Delta_\tau \psi^n,  \varphi_k \right) +  \left(\frac{\partial{\psi^{n+1}}}{\partial{y}}, \frac{\partial\varphi_k }{\partial y} \right) =  &   W^{n+1} \Delta_\tau W^n
\left( y\frac{\partial U^{n+1}}{\partial y}, \varphi_k \right) - h^{n+1}  (h^{n+1})^{\prime} \left( y\frac{\partial u^{n+1}}{\partial y}, \varphi_k \right) \\
\nonumber & +    W^{n+1} {\rm Bi}\left(\frac{b(\tau)}{m_{0}} - {\rm H}U^n(0)\right) \varphi_k(0)\\ 
\nonumber  & -  h^{n+1} {\rm Bi}\left(\frac{b(\tau)}{m_{0}}- {\rm H}u^{n+1}(0)\right) \varphi_k(0) -  \left(\frac{\partial{\rho^{n+1}}}{\partial{y}}, \frac{\partial\varphi_k }{\partial y} \right)\\
 \nonumber & - W^{n+1} \Delta_\tau W^n U^n(1)\varphi_k (1)  +  h^{n+1}  (h^{n+1})^{\prime}u^{n+1}(1)\varphi_k (1) \\ 
\nonumber & +  (h^{n+1})^2 \left( \frac{\partial u^{n+1}}{\partial \tau},  \varphi_k \right) - \displaystyle (W^{n+1})^2 \left( \Delta_\tau(\rho ^n + u^n) ,  \varphi_k \right) \\
\label{I} = :& \sum_{\ell = 1}^5I_{\ell}, 
\end{align}
where we introduce the following notations:
\begin{align*}
&I_1 := W^{n+1} \Delta_\tau W^n
\left( y\frac{\partial U^{n+1}}{\partial y}, \varphi_k \right) - h^{n+1}  (h^{n+1})^{\prime} \left( y\frac{\partial u^{n+1}}{\partial y}, \varphi_k \right),\\
&I_2:=  W^{n+1} {\rm Bi}\left(\frac{b(\tau)}{m_{0}} - {\rm H}U^n(0)\right) \varphi_k(0)  -  h^{n+1} {\rm Bi}\left(\frac{b(\tau)}{m_{0}}- {\rm H}u^{n+1}(0)\right) \varphi_k(0),\\
&I_3:= -  \left(\frac{\partial{\rho^{n+1}}}{\partial{y}}, \frac{\partial\varphi_k }{\partial y} \right),\\
&I_4:=  - W^{n+1} \Delta_\tau W^n U^n(1)\varphi_k (1)  +  h^{n+1}  (h^{n+1})^{\prime}u^{n+1}(1)\varphi_k (1),\\
&I_5:=  (h^{n+1})^2 \left( \frac{\partial u^{n+1}}{\partial \tau},  \varphi_k \right) - \displaystyle (W^{n+1})^2 \left( \Delta_\tau(\rho ^n + u^n),  \varphi_k \right).
\end{align*}
Before proceeding further, we collect two useful estimates in following auxillary Lemma \ref{remark1}. 
\begin{lemma}
	\label{remark1}
	There exist constants $K_2 = K(A_0, m_0)>0$ and $ \hat{c}>0$ such that 
\begin{enumerate}[label = {\rm(\roman*)}] 
\item \label{re1}$\left| \Delta_\tau W^n - (h^{n+1})^{\prime} \right| \leq  K_2 \left(\left| e^n(1)\right| + \left| e_1^n \right|  +  \tilde{c} \Delta \tau   \right), $	
 \item \label{re2} $|e^n(1)| |\varphi(1)|  \leq    \hat{c} \left( (\zeta+ c_{\zeta}) ck^4\left\lVert u^n \right\lVert^2_{H^2} +  (\zeta+ c_{\zeta}) \left\lVert \psi^n \right\lVert^2 + \zeta ck^2  \left\lVert  u^n  \right\lVert^2_{H^2} \right. \\
   \left. \qquad \qquad \;\;\;\;\;\;\;\;\;+ \zeta  \left\lVert \frac{\partial \psi^n}{\partial y}  \right\lVert^2 + (\xi+ c_{\xi}) \left\lVert \varphi \right\lVert^2 + \xi  \left\lVert \frac{\partial \varphi}{\partial y} \right\lVert^2 \right).$
\end{enumerate}
\end{lemma}
\begin{proof}
 Subtracting  \eqref{a20} from \eqref{b2}, we get the following identity:
\begin{align*}
\Delta_\tau W^n - (h^{n+1})^{\prime} & = A_0\left(U^n(1) -\frac{\sigma(W^n)}{m_0}\right) -  A_0\left(u^{n+1}(1)-\frac{\sigma(h^{n+1})}{m_0}\right)\\
& = A_0 \left( U^n(1) - u^{n+1}(1)\right) - \frac{A_0}{m_0} \left( \sigma(W^n) -  \sigma(h^{n+1})\right).
\end{align*}
Using \ref{A4} yields 
\begin{align*}
\left|\Delta_\tau  W^n - (h^{n+1})^{\prime}  \right|\leq& A_0 \left(\left| U^n(1) - u^n(1)  \right| + \left| u^n(1)  - u^{n+1}(1)\right| \right)\\
& + \frac{\mathcal L}{m_0}\left(\left| W^n - h^n \right| + \left| h^n  - h^{n+1}\right| \right)\\
\leq &A_0 \left(\left| e^n(1)\right| +  \Delta \tau  \| \partial_\tau u^n(1)\|_{\infty} \right)
 + \frac{\mathcal L A_0}{m_0}\left(\left| e_1^n \right| + \Delta \tau  \| h^{\prime}\|_{\infty}\right)\\
 \leq & K_2 \left(\left| e^n(1)\right| + \left| e_1^n \right|  +   \Delta \tau {\| \partial_\tau u^n(1)\|_{\infty}} 
 + \Delta \tau  \| h^{\prime}\|_{\infty}\right),
\end{align*}
where $\mathcal L$ is the Lipschitz constant of $\sigma$ in \ref{A4} and  $K_2 := \max\{ A_0, \mathcal{L} \mathcal{A}_0/m_0\}$.  Note that it exists $\hat{c}>0$ such that for $\theta \in [1/2, 1)$ it holds $$\| \partial_\tau u^n(1)\|_{\infty} \leq \hat{c} \| \partial_\tau u^n\|^{\theta}_{L^2(0, 1)} \| \partial_\tau u^n\|^{1-\theta}_{H^1(0, 1)}\leq c $$
 as $u(\tau^n, x) = u^n(x) \in W^{1,2} ((0,T) \times (0, 1))$ (cf. Definition 3.2 in \cite{kumazaki2020global}).
 Thus, we obtain
 \begin{align}
 \left|\Delta_\tau W^n - (h^{n+1})^{\prime}  \right|\leq  K_2 \left(\left| e^n(1)\right| + \left| e_1^n \right|  +  \tilde{c} \Delta \tau   \right),
 \end{align}
 where $\tilde{c} = 2 \max \{c, \|h^{\prime}\|_{\infty}\}$.
This proves  \ref{re1}.\\  We now prove \ref{re2}. Using the interpolation inequality \eqref{a33} yields
\begin{align*}
\nonumber |e^n(1)| |\varphi_k(1)|&\leq \hat{c}  \left\lVert e^n \right\lVert^{\theta}  \left\lVert  e^n \right\lVert_{H^1(0, 1)}^{1-\theta}  \left\lVert \varphi_k \right\lVert^{1- \theta}  \left\lVert  \varphi_k \right\lVert_{H^1(0, 1)}^{1-\theta}\\
&\leq \frac{\hat{c}}{2} \left(\left\lVert e^n \right\lVert \left\lVert  e^n  \right\lVert_{H^1(0, 1)} + \left\lVert \varphi_k \right\lVert  \left\lVert  \varphi_k \right\lVert_{H^1(0, 1)}\right)\\
&\leq  \frac{\hat{c}}{2}  \left( (\zeta+ c_{\zeta}) \left\lVert e^n \right\lVert^2 + \zeta  \left\lVert \frac{\partial e^n}{\partial y}  \right\lVert^2 + (\xi+ c_{\xi}) \left\lVert \varphi_k \right\lVert^2 + \xi  \left\lVert \frac{\partial \varphi_k}{\partial y} \right\lVert^2 \right)\\
&\leq  \frac{\hat{c}}{2}  \left( (\zeta+ c_{\zeta}) \left\lVert \rho^n \right\lVert^2 +  (\zeta+ c_{\zeta}) \left\lVert \psi^n \right\lVert^2 + \zeta  \left\lVert \frac{\partial \rho^n}{\partial y}  \right\lVert^2 + \zeta  \left\lVert \frac{\partial \psi^n}{\partial y}  \right\lVert^2 +(\xi+ c_{\xi}) \left\lVert \varphi_k \right\lVert^2 + \xi  \left\lVert \frac{\partial \varphi_k}{\partial y} \right\lVert^2 \right)\\
&\leq  \hat{c} \left( (\zeta+ c_{\zeta}) \left\lVert \rho^n \right\lVert^2 +  (\zeta+ c_{\zeta}) \left\lVert \psi^n \right\lVert^2 + \zeta  \left\lVert \frac{\partial \rho^n}{\partial y}  \right\lVert^2 + \zeta  \left\lVert \frac{\partial \psi^n}{\partial y}  \right\lVert^2 + (\xi+ c_{\xi}) \left\lVert \varphi_k \right\lVert^2 + \xi  \left\lVert \frac{\partial \varphi_k}{\partial y} \right\lVert^2 \right)\\
&\leq  \hat{c} \left( (\zeta+ c_{\zeta}) ck^4\left\lVert u^n \right\lVert^2_{H^2} +  (\zeta + c_{\zeta}) \left\lVert \psi^n \right\lVert^2 + \zeta ck^2  \left\lVert  u^n  \right\lVert^2_{H^2} \right. \\
 &\left.+ \zeta  \left\lVert \frac{\partial \psi^n}{\partial y}  \right\lVert^2 + (\xi+ c_{\xi}) \left\lVert \varphi_k \right\lVert^2 + \xi  \left\lVert \frac{\partial \varphi_k}{\partial y} \right\lVert^2 \right).
\end{align*}
\end{proof}
We now estimate the terms $I_1-I_5$ as follows.\\
By adding and subtracting appropriate terms, we get
\begin{align*}
I_1 =&  W^{n+1} \Delta_\tau W^n
\left( y\frac{\partial U^{n+1}}{\partial y}, \varphi_k \right) - h^{n+1}(h^{n+1})^{\prime} \left( y\frac{\partial u^{n+1}}{\partial y}, \varphi_k \right)\\
=& (W^{n+1} - h^{n+1})\Delta_\tau W^n \left( y \frac{\partial U^{n+1}}{ \partial y}, \varphi_k \right) + h^{n+1} \left(\Delta_\tau W^n - (h^{n+1})^{\prime}\right) \left( y \frac{\partial U^{n+1}}{ \partial y}, \varphi_k \right)\\
& +  h^{n+1} (h^{n+1})^{\prime} \left( y \frac{\partial }{ \partial y}  (U^{n+1}- u^{n+1}), \varphi_k \right).
\end{align*}
Using Lemma \ref{remark1}, it yields
\begin{align*}
\nonumber |I_1| \leq& |e_1^{n+1}| |\Delta_\tau W^n|  \left\lVert \frac{\partial U^{n+1} }{\partial y} \right\lVert \left\lVert \varphi_k  \right\lVert  +  \tilde{K}_0 K_2 \left(\left| e^n(1)\right| + \left| e_1^n \right|  +  \tilde{c} \Delta \tau  \right)  \left\lVert \frac{\partial U^{n+1} }{\partial y} \right\lVert \left\lVert \varphi_k  \right\lVert  \\
&\nonumber + \tilde{K}_0 \left(\left\lVert \frac{\partial \psi^{n+1} }{\partial y}  \right\lVert+ \left\lVert  \frac{\partial \rho^{n+1} }{\partial y} \right\lVert \right) \left\lVert \varphi_k  \right\lVert. \\
 \leq &K(K^*, K_1)   \left( \xi |e_1^{n+1}|^2 + c_{\xi}  \left\lVert \varphi_k \right\lVert^2 \right)+ K( \tilde{K}_0, K^*, K_2,   \hat{c}) \left(c_{\bar{\xi}} \|\varphi_k\|^2 + (\zeta + c_{\zeta}) \bar{\xi}  \left\lVert e^n \right\lVert^{2} + \zeta \bar{\xi}  \left\lVert \frac{\partial e^n}{\partial y} \right\lVert^{2}\right) \\
& + K(\tilde{K}_0, K_2, K^*, \tilde{c})\left( \bar{\xi}  |e_1^n|^2 + c_{\bar{\xi}} \|\varphi_k\|^2 + \bar{\xi} \Delta \tau^2 +  c_{\bar{\xi}} \|\varphi_k\|^2\right) \\
& + \tilde{K}_0 \left( \xi c k^2 \| u^{n+1} \|_{H^2}^2 + 2 c_{\xi} \|\varphi_k\|^2 + \xi \left\lVert \frac{\partial \psi^{n+1} }{\partial y} \right\lVert^2 \right)  \\
 \leq & K(K^*, K_1)  \left( \xi |e_1^{n+1}|^2 + c_{\xi}  \left\lVert \varphi_k \right\lVert^2 \right) \\
 &+ K(\tilde{K}_0, K^*, K_2,   \hat{c}) \left(c_{\bar{\xi}} \|\varphi_k\|^2 + (\zeta + c_{\zeta}) \bar{\xi} ck^4  \left\lVert u^n \right\lVert^{2} + (\zeta + c_{\zeta}) \bar{\xi} \| \psi^n\|^2 + \zeta \bar{\xi} ck^2  \left\lVert u^n \right\lVert^{2} + \zeta \bar{\xi}  \left\lVert \frac{\partial \psi^n}{\partial y} \right\lVert^{2}\right)\\
& + K(\tilde{K}_0, K_1, K^*, \tilde{c})\left( \bar{\xi}  |e_1^n|^2 + c_{\bar{\xi}} \|\varphi_k\|^2 + \bar{\xi} \Delta \tau^2 +  c_{\bar{\xi}} \|\varphi_k\|^2\right)\\
& + \tilde{K}_0 \left(\xi c k^2 \| u^{n+1} \|_{H^2} + c_{\xi} \|\varphi_k\|^2 + \xi \left\lVert \frac{\partial \psi^{n+1} }{\partial y} \right\lVert^2 \right). 
\end{align*}
After re-arranging the term conveniently,  $I_2$ becomes
\begin{align}
\nonumber   I_2&= W^{n+1} {\rm Bi}\left(\frac{b(\tau)}{m_{0}} - {\rm H}U^n(0)\right) \varphi_k(0) -  h^{n+1}{\rm Bi}\left(\frac{b(\tau)}{m_{0}}- {\rm H}u^{n+1}(0)\right) \varphi_k(0)\\
 \nonumber &= {\rm Bi}\frac{b(\tau)}{m_{0}} (W^{n+1} - h^{n+1}) \varphi_k(0) - {\rm Bi}  {\rm H} (W^{n+1} - h^{n+1}) U^n(0) \varphi_k(0)\\
\nonumber  & -  {\rm Bi}  {\rm H} h^{n+1} (U^n(0) - u^{n+1}(0))\varphi_k(0).
\end{align}
It holds
\begin{align}
 \nonumber |I_2|&\leq K_3(b^*, m_0, \text{Bi}, \text{H}, \tilde{K}_0) \left( |e_1^{n+1}| |\varphi(0)| + |e^n(0)| |\varphi(0)| + |u^n(0) - u^{n+1}(0)||\varphi(0)| \right) \\
 \nonumber & \leq K_3(b^*, m_0, \text{Bi}, \text{H}, \tilde{K}_0) \left(c_{\xi} |e_1^{n+1}|^2 + (\bar{\xi} + c_{\bar{\xi}}) \xi  \left\lVert \varphi_k \right\lVert^{2} + \xi \bar{\xi}  \left\lVert \frac{\partial \varphi_k}{\partial y} \right\lVert^{2} + (\zeta+ c_{\zeta}) ck^4\left\lVert u^n \right\lVert^2_{H^2} \right. \\
 \nonumber &\left. +  (\zeta + c_{\zeta}) \left\lVert \psi^n \right\lVert^2 + \zeta ck^2  \left\lVert  u^n  \right\lVert^2_{H^2} + \zeta  \left\lVert \frac{\partial \psi^n}{\partial y}  \right\lVert^2 +(\xi + c_{\xi}) \left\lVert \varphi_k \right\lVert^2 + \xi  \left\lVert \frac{\partial \varphi_k}{\partial y} \right\lVert^2 + \xi \tilde{c} \Delta \tau^2\right).
 \end{align}
 The bound on $|I_3|$ follows from the Cauchy-Schwartz inequality,  Young's inequality and Lemma \ref{lemma1}
 \begin{align}
\nonumber |I _3| &   \leq \left| \left(\frac{\partial{\rho^{n+1}}}{\partial{y}}, \frac{\partial\varphi_k }{\partial y} \right) \right| \leq \left\lVert \frac{\partial \rho^{n+1} }{\partial y}  \right\lVert \left\lVert \frac{\partial \varphi_{k} }{\partial y}  \right\lVert \leq c_{\xi} \gamma_2^2 k^4 \|u^{n+1}\|_{H^2}^2 + \xi \left\lVert \frac{\partial \varphi_{k} }{\partial y}  \right\lVert^2.
\end{align}
To deal with $I_4$, we start by re-arranging the term in a more convenient way
\begin{align}
\nonumber I_4 &=   - W^{n+1} \Delta_\tau W^n U^n(1)\varphi_k (1)  +  h^{n+1}  (h^{n+1})^{\prime}u^{n+1}(1)\varphi_k (1)\\
\nonumber  &= -(W^{n+1} - h^{n+1})  \Delta_\tau W^n U^n(1) \varphi_k(1) - h^{n+1} (\Delta_{\tau} W^n - (h^{n+1})^{\prime})U^n(1) \varphi_k(1) \\
\nonumber &- h^{n+1} (h^{n+1})^{\prime} (U^n(1) - u^{n+1}(1)) \varphi_k(1) = I_{4,1} + I_{4,2} + I_{4,3}. 
\end{align}
Using Lemma \ref{remark1}, we obtain
\begin{align*}
|I_{4,1}|&\leq  \tilde{K}|e^{n+1}_1| |U^n(1)| |\varphi(1)| \leq \tilde{K} \hat{c}^2 |e^{n+1}_1|  \left\lVert U^n \right\lVert^{\theta}  \left\lVert  U^n \right\lVert_{H^1(0, 1)}^{1- \theta} \left\lVert \varphi_k \right\lVert^{\theta}  \left\lVert \varphi_k \right\lVert_{H^1(0, 1)}^{1-\theta} \\
&\leq K(K^*, \tilde{K},\hat{c}) \left(c_{\bar{\xi}} |e^{n+1}_1|^2 + (\xi + c_{\xi}) \bar{\xi}  \left\lVert \varphi_k \right\lVert^{2} + \xi \bar{\xi}  \left\lVert \frac{\partial \varphi_k}{\partial y} \right\lVert^{2}\right)\\
|I_{4,2}|& \leq |h^{n+1}| |\Delta_{\tau} W^n - (h^{n+1})^{\prime}| |U^n(1)| |\varphi_k(1)| \\
& \leq K(\tilde{K}_0, K_2) \left(\left| e^n(1)\right| + \left| e_1^n \right|  +   \tilde{c} \Delta \tau  \right) |U^n(1)| |\varphi_k(1)| \\
& \leq K(\hat{c}, K^*, \tilde{K}_0, K_2) \left((\zeta + c_{\zeta}) k^4\left\lVert u^n \right\lVert^2_{H^2} +  (\zeta + c_{\zeta}) \left\lVert \psi^n \right\lVert^2 + \zeta k^2  \left\lVert  u^n  \right\lVert^2_{H^2}  + \zeta  \left\lVert \frac{\partial \psi^n}{\partial y}  \right\lVert^2 \right.\\
& \left. +(\xi + c_{\xi}) \left\lVert \varphi_k \right\lVert^2 + \xi  \left\lVert \frac{\partial \varphi_k}{\partial y} \right\lVert^2 + c_{\bar{\xi}} |e^{n}_1|^2  + \bar{\xi} (\xi+ c_{\xi}) \| \varphi_k\|^2 + \bar{\xi} \xi \left\lVert \frac{\partial \varphi_k}{\partial y} \right\lVert^2 + c_{\bar{\xi}} \tilde{c}^2 \Delta \tau^2 \right)\\
|I_{4,3}|& \leq  \tilde{K}_0 |U^n(1) - u^{n+1}(1)| |\varphi_k(1)| \\
& \leq \tilde{K}_0 \left( |e^n(1)| + |u^n(1) - u^{n+1}(1)|\right)|\varphi_k(1)| \\
& \leq \tilde{K}_0 \left( |e^n(1)| + c\Delta \tau \right)|\varphi_k(1)| \\
& \leq K(\tilde{K}_0) \left( (\zeta + c_{\zeta}) k^4\left\lVert u^n \right\lVert^2_{H^2} +  (\zeta + c_{\zeta}) \left\lVert \psi^n \right\lVert^2 + \zeta k^2  \left\lVert  u^n  \right\lVert^2_{H^2}  + \zeta  \left\lVert \frac{\partial \psi^n}{\partial y}  \right\lVert^2 \right.\\
& \left. +(\xi + c_{\xi}) \left\lVert \varphi_k \right\lVert^2 + \xi  \left\lVert \frac{\partial \varphi_k}{\partial y} \right\lVert^2 + c_{\xi} c^2 \Delta \tau^2\right).
\end{align*}
 By adding and subtracting appropriate terms in $I_5$, we get
\begin{align*}
\nonumber  I_5 &=  (h^{n+1})^2 \left( \frac{\partial u^{n+1}}{\partial \tau},  \varphi_k \right) - \displaystyle (W^{n+1})^2 \left( \Delta_t(\rho ^n + u^n) ,  \varphi_k \right)\\
\nonumber &=  \left((h^{n+1})^2 -  (W^{n+1})^2 \right) \left( \frac{\partial u^{n+1}}{\partial \tau},  \varphi_k \right) \\
 & + (W^{n+1})^2 \left( \frac{\partial u^{n+1}}{\partial \tau} - \Delta_\tau u^n, \varphi_k \right) - (W^{n+1})^2 (\Delta_\tau \rho^n, \varphi_k).
\end{align*}
We now claim the following holds:
By the Taylor expansion of $u^{n+1}$  around $\tau^n$ with integral reminder yields
\begin{align}
\label{taylor}u^n = u^{n+1} - \Delta\tau \frac{\partial u^{n+1}}{\partial \tau} + \int_{\tau^{n+1}}^{\tau^{n}} (\tau^n - s) \partial_{\tau \tau} u (s)ds.
\end{align}
To prove \eqref{taylor},  the fundamental theorem of calculus gives 
\begin{align}
\label{fun}u^{n} = u^{n+1} - \int_{\tau^n}^{\tau^{n+1}} u^{\prime}(s)ds.
\end{align}
Integrating by parts in the last term of \eqref{fun} gives
\begin{align}
\label{fun2}u^{n} = u^{n+1} - \left[\tau^{n+1} \frac{\partial u^{n+1}}{\partial \tau} - \tau^{n} \frac{\partial u^{n}}{\partial \tau} -  \int_{\tau^n}^{\tau^{n+1}} u^{\prime \prime}(s) s ds \right].
\end{align}
Using  again the fundamental  theorem of calculus for the last but one term in \eqref{fun2} leads to
\begin{align*}
u^{n} = u^{n+1} - \left[\tau^{n+1} \frac{\partial u^{n+1}}{\partial \tau} - \tau^{n} \left( \frac{\partial u^{n+1}}{\partial \tau} -  \int_{\tau^n}^{\tau^{n+1}} u^{\prime \prime}(s) ds \right) -  \int_{\tau^n}^{\tau^{n+1}} u^{\prime \prime}(s) s ds \right].
\end{align*}
This proves \eqref{taylor}.
It is worth mentioning that \eqref{taylor} resembles the application of Taylor's approximation  with integral reminder for the function $u^n$ around $\tau^{n+1}$ under the assumption $u^{\prime\prime} \in L^1(\tau^n, \tau^{n+1})$ (see Theorem 1.3 in \cite{cobzas1997calculul}).  
With the help of \eqref{taylor} and Preposition \ref{higher},  we can estimate the  second term as follows:
\begin{align}
\nonumber(W^{n+1})^2 \left( \frac{\partial u^{n+1}}{\partial \tau} - \Delta_\tau u^n, \varphi_k \right) &= (W^{n+1})^2 \left( \frac{1}{\Delta \tau} \int_{\tau^n}^{\tau^{n+1}} (s- \tau^n)  \partial_{\tau \tau} u(s)ds,  \varphi_k\right) \\
\nonumber&\leq  \frac{K_1}{\Delta \tau} \left\Vert \int_{\tau^n}^{\tau^{n+1}} (s- \tau^n)  \partial_{\tau \tau} u(s)ds \right\Vert \left\Vert \varphi_k \right\Vert \\ 
\nonumber&\leq  \frac{K_1}{\Delta \tau} \left\Vert \sup_{[\tau^n, \tau^{n+1}]} |\partial_{\tau \tau} u| \int_{\tau^n}^{\tau^{n+1}} (s- \tau^n) ds \right\Vert \left\Vert \varphi_k \right\Vert \\ 
\nonumber & \leq K_1 |\Omega|^{\frac{1}{2}} \frac{\Delta \tau}{2}   \sup_{[\tau^n, \tau^{n+1}]} |\partial_{\tau \tau} u| \| \varphi_k\|\\
 \nonumber & \leq K_1 \mathcal{K}_0 \Delta \tau \| \varphi_k\|\\
\label{3a57}& \leq   \xi \| \varphi_k\|^2 + c_{\xi} K_1^2 \mathcal{K}_0^2 \Delta \tau^2.
\end{align}
The last but one  term in $I_5$ can be estimated as follows:
\begin{align}
\nonumber | (W^{n+1})^2 (\Delta_\tau \rho^n, \varphi_k)| &\leq \frac{K_1}{\Delta \tau} \left|\left(\rho^{n+1}-\rho^n, \varphi_k \right) \right|\\
\nonumber&\leq  \frac{K_1}{\Delta \tau} \left\Vert \rho^{n+1}-\rho^n \right\Vert \left\Vert  \varphi_k \right\Vert\\
\nonumber& =  \frac{K_1}{\Delta \tau} \left\Vert  \int_{\tau^n}^{\tau^{n+1}} \partial_\tau\rho(s)ds \right\Vert  \left\Vert  \varphi_k \right\Vert\\
\nonumber & \leq C(u) k^2 \left\Vert  \varphi_k \right\Vert\\
\nonumber & \leq  \xi \left\Vert  \varphi_k \right\Vert^2 + c_{\xi} C(u) k^4. 
\end{align}
Finally, we bound $I_5$  by 
\begin{align}
\nonumber |I_5|&\leq K(K_1, \tilde{K_0}) |e_1^{n+1}|   \left\Vert  \varphi_k \right\Vert + \xi \| \varphi_k\|^2 + c_{\xi} K_1^2 \mathcal{K}_0^2 \Delta \tau^2 + \xi \left\Vert  \varphi_k \right\Vert^2 + c_{\xi} C(u) k^4\\
\nonumber &\leq K(K_1, \tilde{K_0}, \mathcal{K}_0^2) \left( c_{\xi}|e_1^{n+1}|^2  + 3 \xi \| \varphi_k\|^2 + c_{\xi}  \Delta \tau^2  + c_{\xi} C(u) k^4\right)
\end{align}
By using $2(b-a, b) = b^2 - a^2 + (b-a)^2$, we also note the following estimate holds:
\begin{align}
\label{I1}&\left( \Delta_{\tau} \psi^n,  \psi^{n+1} \right) \geq \frac{1}{2 \Delta \tau}  \left\{ \| \psi^{n+1} \|^2 -  \| \psi^{n} \|^2\right\}\\
\label{I2}&\Delta_{\tau} e_1^n e_1^{n+1} \geq \frac{1}{2 \Delta \tau}  \left\{ | e_1^{n+1} |^2 -  | e_1^{n} |^2\right\}.
\end{align}
We now consider the equations corresponding to the position of the moving boundary. We write 
\begin{align}
  \label{error2} \Delta_\tau e_1^n =  \Delta_\tau W^n -  \Delta_\tau h^n
\end{align}
Multiplying  \eqref{error2} by $e^{n+1}_1$ gives
\begin{align}
\nonumber \Delta_\tau e_1^n e^{n+1}_1 &=  (\Delta_\tau W^n - \Delta_\tau h^n) e^{n+1}_1 \\
\label{b7}&= [\Delta_\tau W^n  - (h^{n+1})^{\prime}]e_1^{n+1} +[ (h^{n+1})^{\prime}  - \Delta_\tau h^n] e^{n+1}_1. 
\end{align}
Using \eqref{I2} in \eqref{b7}, we get
\begin{align}
\label{identity}\frac{1}{2 \Delta \tau}  \left\{ | e_1^{n+1} |^2 -  | e_1^{n} |^2\right\} \leq |\Delta_\tau W^n  - (h^{n+1})^{\prime}|  |e^{n+1}_1|   + |(h^{n+1})^{\prime}  - \Delta_\tau h^n| |e^{n+1}_1|. 
\end{align}
By using Taylor series expansion and Preposition \ref{higher}, we estimate the  second term in \eqref{identity} as follows:
\begin{align}
\nonumber |(h^{n+1})^{\prime}  - \Delta_\tau h^n| |e^{n+1}_1|&
 \leq \frac{\Delta \tau}{2} \sup_{[\tau^n, \tau^{n+1}]} |h^{\prime \prime}(\tau)||e^{n+1}_1|\\
\label{b8} & \leq c_{\bar{\xi}} |e^{n+1}_1|^2 +  \bar{\xi}\Delta \tau^2 \mathcal{K}_1^2. 
\end{align}
Using Lemma \ref{remark1}, \eqref{b8} and Young's inequality in \eqref{identity} yields
\begin{align}
\nonumber \frac{1}{2 \Delta \tau}  \left\{ | e_1^{n+1} |^2 -  | e_1^{n} |^2\right\} \leq &K(\tilde{K}_0, K_2,  \hat{c}) \left(c_{\bar{\xi}} |e_1^{n+1}|^2 + (\zeta + c_{\zeta}) \bar{\xi} k^4  + (\zeta + c_{\zeta}) \bar{\xi} \| \psi^n\|^2 + \xi \bar{\xi} k^2   +  \xi \bar{\xi}  \left\lVert \frac{\partial \psi^n}{\partial y} \right\lVert^{2}\right)\\
\label{b10}& + K(K_1, K_2, \tilde{c})\left(\bar{\xi}  |e_1^n|^2 + c_{\bar{\xi}} | e_1^{n+1} |^2 + \bar{\xi} \Delta \tau^2 \right).
\end{align}
Taking $\varphi_k = \psi^{n+1}$ in  \eqref{I} and adding the result to \eqref{b10}, we get the following estimate:
\begin{align}
\nonumber (W^{n+1})^2\frac{1}{2 \Delta \tau}  \left\{ \| \psi^{n+1} \|^2 -  \| \psi^{n} \|^2\right\} & + \frac{1}{2 \Delta \tau}  \left\{ | e_1^{n+1} |^2 -  | e_1^{n} |^2\right\} + \left\lVert \frac{\partial \psi^{n+1} }{\partial y}  \right\lVert^2 \\
\nonumber & \leq K_3(\tilde{c}, \hat{c}, \xi, \bar{\xi}, c_{\xi}, c_{\bar{\xi}},  \tilde{K}_0) \left( \Delta \tau^2 + k^2 + k^4 \right) \\
\nonumber& +K_4(\tilde{c}, \hat{c}, \xi, \bar{\xi}, c_{\xi}, c_{\bar{\xi}},   \tilde{K}_0, K^*) \left( \left\lVert \psi^n \right\lVert^2 + \left\lVert \psi^{n+1} \right\lVert^2 +  |e^{n}_1|^2  + |e_1^{n+1}|^2 \right) \\
\label{inequality}& + K_5(\zeta, \bar{\xi})  \left\lVert \frac{\partial \psi^n}{\partial y}  \right\lVert^2  + K_6(\xi, \bar{\xi})  \left\lVert \frac{\partial \psi^{n+1}}{\partial y} \right\lVert^2.
\end{align}
Adding $(1 - K_6(\xi, \bar{\xi}) - K_5(\zeta, \bar{\xi}))\left\lVert \frac{\partial \psi^n}{\partial  y}\right\lVert $ to both sides of \eqref{inequality} and multiplying the result by $2\Delta \tau$, one gets
\begin{align}
\nonumber  \| \psi^{n+1} \|^2 -  \| \psi^{n} \|^2 +  & | e_1^{n+1} |^2 -  | e_1^{n} |^2 +  2 \Delta \tau (1 - K_6(\xi, \bar{\xi} ) \left\lVert \frac{\partial \psi^{n+1} }{\partial y}  \right\lVert^2 + 2 \Delta\tau (1 - K_6(\xi, \bar{\xi}) - K_5(\zeta, \bar{\xi})) \left\lVert \frac{\partial \psi^{n} }{\partial y}  \right\lVert^2  \\
\nonumber & \leq \Delta \tau K_3(\tilde{c}, \hat{c}, \xi, \tilde{K}_0) \left( \Delta \tau^2 + k^2 \right) + 2 \Delta \tau (1 - K_6(\xi, \bar{\xi})) \left\lVert \frac{\partial \psi^{n} }{\partial y}  \right\lVert^2 \\
\label{inequality1}&+ \Delta \tau K_4(\tilde{c}, \hat{c}, \xi, \bar{\xi}, \tilde{K}_0, K^*) \left( \left\lVert \psi^n \right\lVert^2 + \left\lVert \psi^{n+1} \right\lVert^2
  +  |e^{n}_1|^2   +  |e_1^{n+1}|^2 \right). 
\end{align}
With the notation: $c_n:= \| \psi^{n} \|^2  +   | e_1^{n} |^2,\;  d_n := \left\lVert \frac{\partial \psi^{n} }{\partial y}  \right\lVert^2,\; E =  \Delta \tau^2 + k^2,\; \alpha = 2(1 - K_6(\xi, \bar{\xi}) - K_5(\zeta, \bar{\xi})), \gamma =  2(1 - K_6(\xi, \bar{\xi})) $
the inequality \eqref{inequality1} can be rewritten as
\begin{align}
\label{c1}(1 - \Delta \tau K_4) c_{n+1} +  \gamma \Delta \tau d_{n+1} + \alpha \Delta \tau d_{n} \leq (1 + \Delta \tau K_4) c_{n} + \gamma \Delta \tau d_{n} + \Delta \tau K_3 E.
\end{align}
By dividing both sides of \eqref{c1} by $(1 - \Delta \tau K_4)$, we get the following inequalities:
\begin{align}
\label{c3}c_{n+1} + a_1  d_{n+1} + a_2 d_{n} &\leq a_3 ( c_{n} + a_1 d_{n} )+ a_4 E,\\
\label{c4}a_3(c_n + a_1 d_n) + a_3 a_2 d_{n-1} &\leq a_3^2  (c_{n-1} + a_1 d_{n-1} )+ a_3 a_4 E,\\
\label{c5}a_3^2(c_{n-1} + a_1 d_{n-1}) + a_3^2 a_2 d_{n-2} &\leq a_3^3  (c_{n-2} + a_1 d_{n-2} )+ a_3^2 a_4 E,\\
\nonumber&\vdots\\
\label{c6}a_3^n(c_{1} + a_1 d_{1}) + a_3^n a_2 d_{0} &\leq a_3^{n+1}  (c_{0} + a_1 d_{0} )+ a_3^n a_4 E,
\end{align}
where
\begin{align*}
a_1= \frac{\gamma \Delta \tau}{1 - \Delta \tau K_4},\; a_2= \frac{\alpha \Delta \tau}{1 - \Delta \tau K_4},\;  a_3 = \frac{1 + \Delta \tau K_4}{1 - \Delta \tau K_4},\; a_4 = \frac{\Delta \tau K_3}{1 - \Delta \tau K_4}.
\end{align*}
Adding \eqref{c3}-\eqref{c6}, we obtain
\begin{align*}
c_{n+1} + a_1  d_{n+1} + a_2 [d_{n} + a_3 d_{n-1} + a_3^2 d_{n-2} + \cdots +  a_3^n d_0] \leq a_4 E [1 + a_3 + a_3^2+ \cdots + a_3^n ].
\end{align*}
For sufficiently small $\Delta \tau$, we can write 
\begin{align}
 \label{b9}\| \psi^{n+1} \|^2  +   | e_1^{n+1} |^2 + \alpha \Delta \tau \sum_{i=1}^{n+1}  \left\lVert \frac{\partial \psi^{j} }{\partial y}  \right\lVert^2 \leq K_7\{\Delta \tau ^2 + k^2\}.
\end{align}
This completes the proof. 
\end{proof}
\begin{theorem}
Assume the assumption of Theorem \ref{theorem1} holds.	Then there exists a constant $K_8 = K(K_7, \tilde{K}_0) >0$ such that 
	\begin{align*}
	\max_{0\leq n\leq M} \| e^n\|^2 + \max_{0\leq n\leq M} | e_1^n|^2\leq K_6 \{ \Delta \tau^2 + k^2 \}. 
	\end{align*}
\end{theorem}
\begin{proof} Relying on Theorem \ref{theorem1} and using the relation $(a+b)^2\leq 2(a^2 + b^2)$, it holds for all $0\leq n\leq M$,
\begin{align*}
\| e^n\|^2 = \| \psi^n + \rho^n\|^2 & \leq 2(\| \psi^n\|^2 + \|\rho^n \|^2)\\
& \leq  K_7 \{ \Delta \tau^2 + k^2 \} + c \tilde{K}_0 k^2\\
&\leq  K_8(K_7, \tilde{K}_0) \{ \Delta \tau^2 + k^2 \},
\end{align*}
where $K_7$ is as used in \eqref{b9}.
\end{proof}
\begin{remark}
If we assume only $u_0 \in H^2(0, s_0)$ and $b\in W^{1,2}(0, T)$, then we can only prove at most $\sqrt{\Delta \tau}$ convergence rate in time. This fact is a consequence of handling the estimate \eqref{3a57}.
\end{remark}
\section{Numerical experiments}\label{experiments}
In this section, we present  numerical results to substantiate the theoretically obtained order of convergence in space and time proposed in Section \ref{fully}.  We solve \eqref{a19}--\eqref{a22} by using the method of lines; for more details see, for instance, \cite{larsson2008partial}.   
We refer the interested  reader to \cite{nepal2021moving}  for more technical  information on experimental data, implementation of the numerical method, and additional simulation results. The reader may also consult  \cite{nepal2021error} for \textit{a priori} and \textit{a posteriori} error estimates of our semi-discrete finite element approximation.  
Here we take the final time $T_f$ to be $10$ minutes. We discretize the domain in a uniform mesh size and use piecewise linear functions as basis for the subspace $V_k$. 
The values of parameters are taken to be $s_0 = 0.01$ (mm), $m_0$ = 0.1 (gram/mm$^3$) and  $b=1$ (gram/mm$^3$). We take the value $3.66 \times 10^{-4}$ (mm$^2$/min) for the diffusion constant $D$, $0.564$ (mm/min) for absorption rate $\beta$ and $2.5$ for Henry's constant \text{H}.   We choose $\sigma(s(t)) = s(t)/10$ (gram/mm$^3$) and $a_0 = 50$ (mm$^4$/sec/gram).  This specific choice of  parameters is taken from  \cite{nepal2021moving, nepal2021error}.\\
To test numerically the  convergence order in space, we fix the uniform time step size $\Delta t=0.0001$. As we do not know the exact solution, we compute the finite element approximation on a fine mesh size with the total number of node $N$ to be $1280$ for a reference solution. To compute errors,  we compare our reference solution  against  finite element approximations corresponding to different mesh sizes with increasing total number of nodes as $20, 40, 80, 160, 320$, and  $640$. We calculate the convergence order in space based on any two consecutive calculations of discrete errors. The obtained  errors and
convergence orders in space are  listed in Table \ref{tab:title32}.

\begin {table}[h]
\begin{center}
	\begin{tabular}{ |p{1cm}|p{3.7cm}|p{3cm}|p{2.7cm}|p{3cm}| }
		\hline
		$N$ & $\displaystyle\max_{0\leq n \leq M}\|U^n - u^n\|_{L^2(0, 1)}$& Convergence order& $\displaystyle\max_{0\leq n \leq M}|W^n - h^n|$ & Convergence order \\
		\hline
		20 & 0.5941833& 1.017 & 0.4859140 &0.825\\
		40 & 0.2934375 &1.049 & 0.2741301 & 1.020\\
		 80 &0.1417237& 1.103 & 0.1351390 & 1.110\\
		160  & 0.0659733& 1.224 & 0.0626064& 1.234\\
		320  & 0.0282283& 1.586& 0.0266161 & 1.592\\
		640  & 0.0094013&  & 0.0088266& \\
		\hline
	\end{tabular}
	\caption {Errors and convergence orders in space  with fixed $\Delta t = 10^{-4}$.}
	\label{tab:title32} 
\end{center}
\end {table}
\begin{figure}[ht] 
	\centering
	\includegraphics[width=0.43\textwidth]{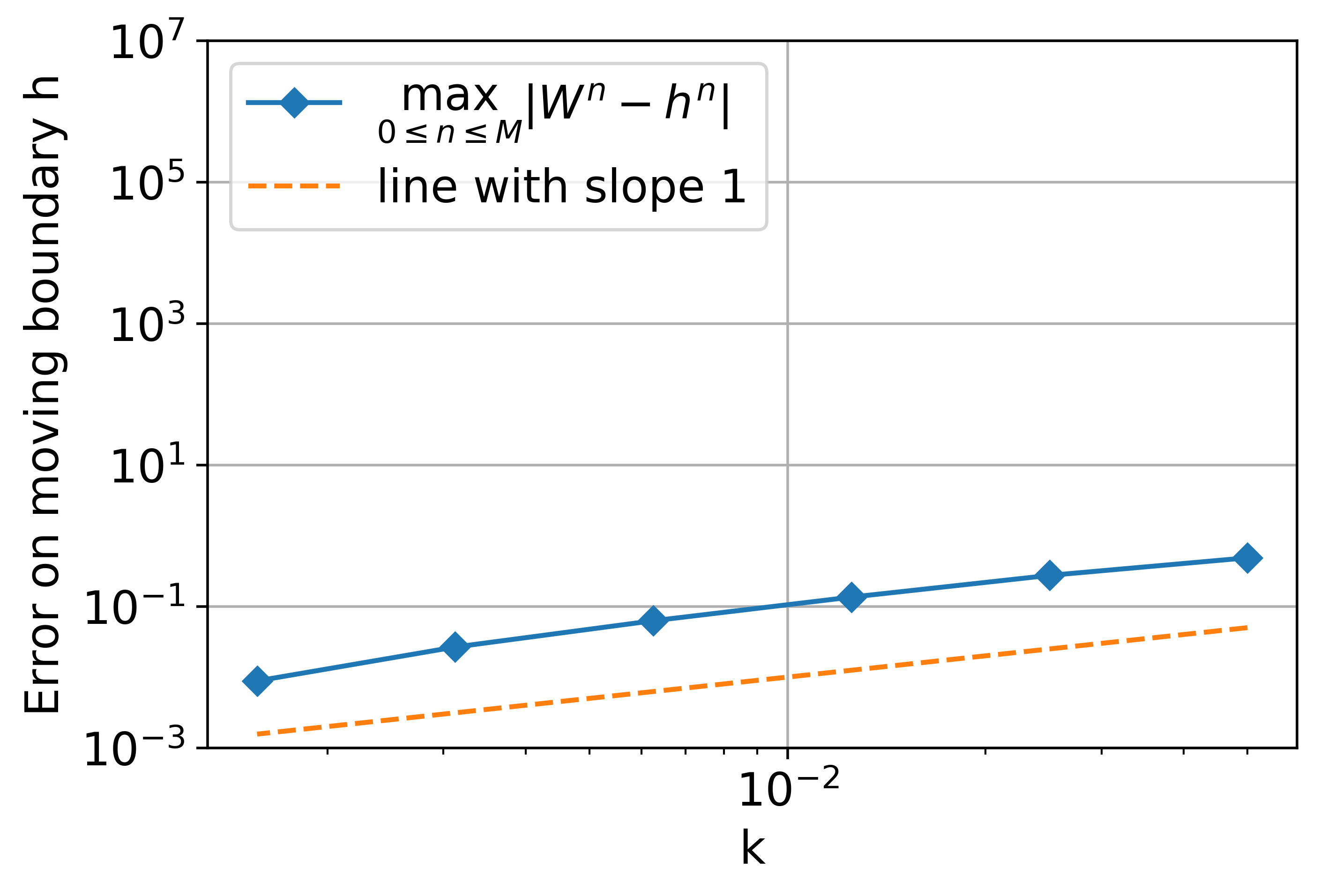}
	\hspace{0.1cm}
	\includegraphics[width=0.43\textwidth]{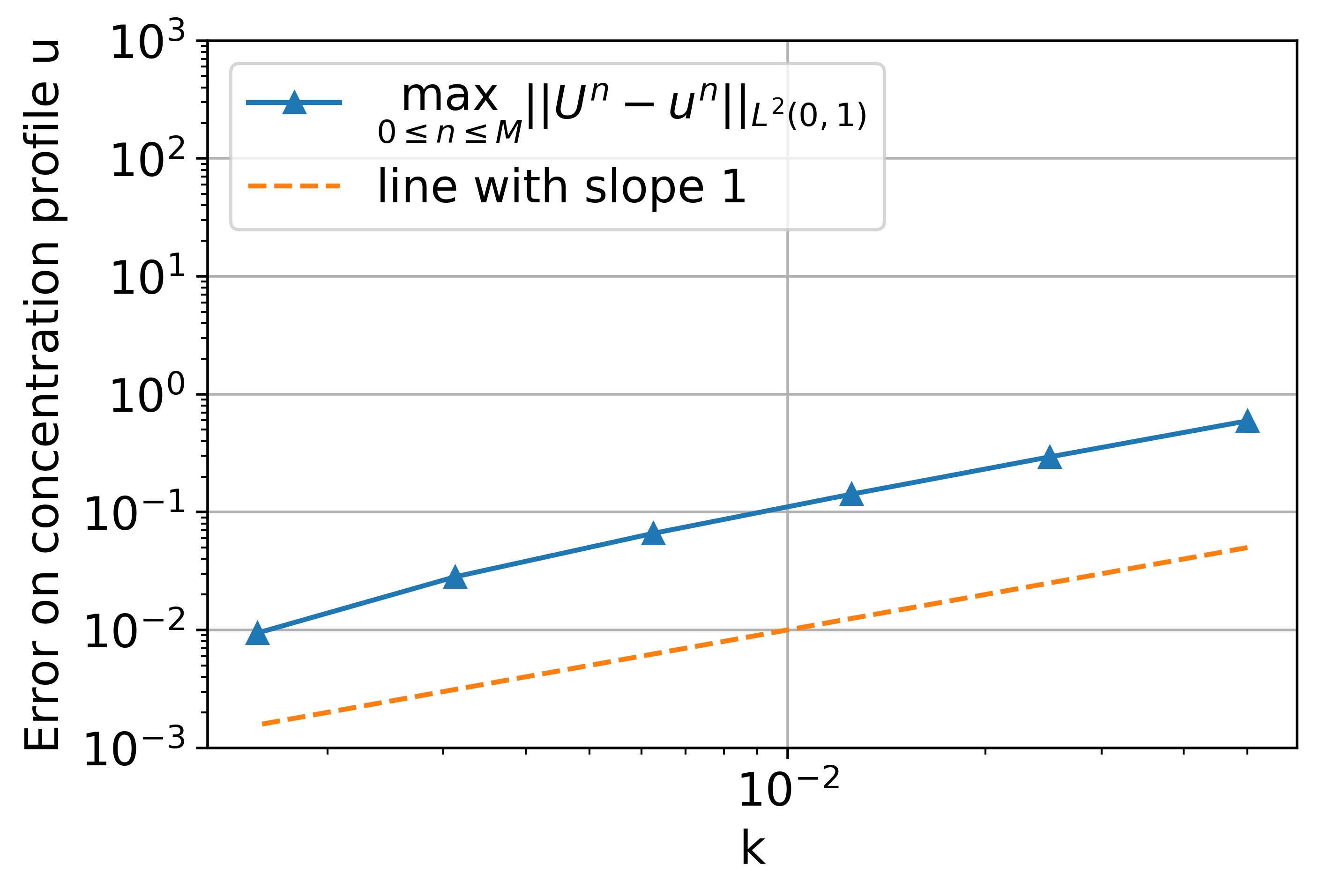}
	\caption{Convergence order in space when time step size $\Delta t = 10^{-4}$  is fixed. Dash lines are lines of slope 1.  Left: Log log scale plot of error on the boundary $\max_{0\leq n \leq M}|W^n - h^n|$.  Right: Log log scale plot of error on the concentration $\max_{0\leq n \leq M}\|U^n - u^n\|_{L^2(0, 1)}$.}
	\label{PlotlogLogfullySpace}
\end{figure}
We show in Figure \ref{PlotlogLogfullySpace} the computed convergence order in space for the approximation  of the moving boundary position and of the concentration profile.\\

To capture numerically the convergence order in time, we fix the total number of space node $N$ to be $320$  and choose $\Delta t = 10^{-3}$. We compute the finite element approximation on a fine time step size $\bar{\Delta t} = \Delta t/64$ for a reference solution. We then compute approximations for different time step sizes $\Delta t, \Delta t/2, \Delta t/4, \Delta t/8, \Delta t/16, \Delta t/32$  and then calculate the  order of convergence in time. The errors and
convergence orders in time are  listed in Table \ref{tab:title33}. We show in Figure \ref{PlotlogLogfullyTime} the computed convergence order in time for the approximation  of both the position of the moving boundary and corresponding concentration profile.
\begin {table}[h]
\begin{center}
	\begin{tabular}{ |p{1.6cm}|p{3.6cm}|p{2.9cm}|p{2.6cm}|p{2.9cm}| }
		\hline
		$\Delta t$ & $\displaystyle\max_{0\leq n \leq M}\|U^n - u^n\|_{L^2(0, 1)}$& Convergence order& $\displaystyle\max_{0\leq n \leq M}|W^n - h^n|$ & Convergence order \\
		\hline
		0.001  & 0.0000229& 1.022 & 0.0000541 &1.009\\
		 0.0005 & 0.0000113 &1.046 & 0.0000268 & 1.085\\
		 0.00025 &0.0000054& 1.099 & 0.0000126 & 1.087\\
		 0.000125 &0.0000025& 1.223 & 0.0000026& 1.231\\
		 0.0000625  & 0.0000010& 1.597& 0.0000025 & 1.349\\
		 0.00003125  & 0.0000003&  & 0.0000009& \\
		\hline
	\end{tabular}
	\caption {Errors and convergence orders in time  with fixed $N = 320$.}
	\label{tab:title33} 
\end{center}
\end {table}


\begin{figure}[ht] 
	\centering
	\includegraphics[width=0.43\textwidth]{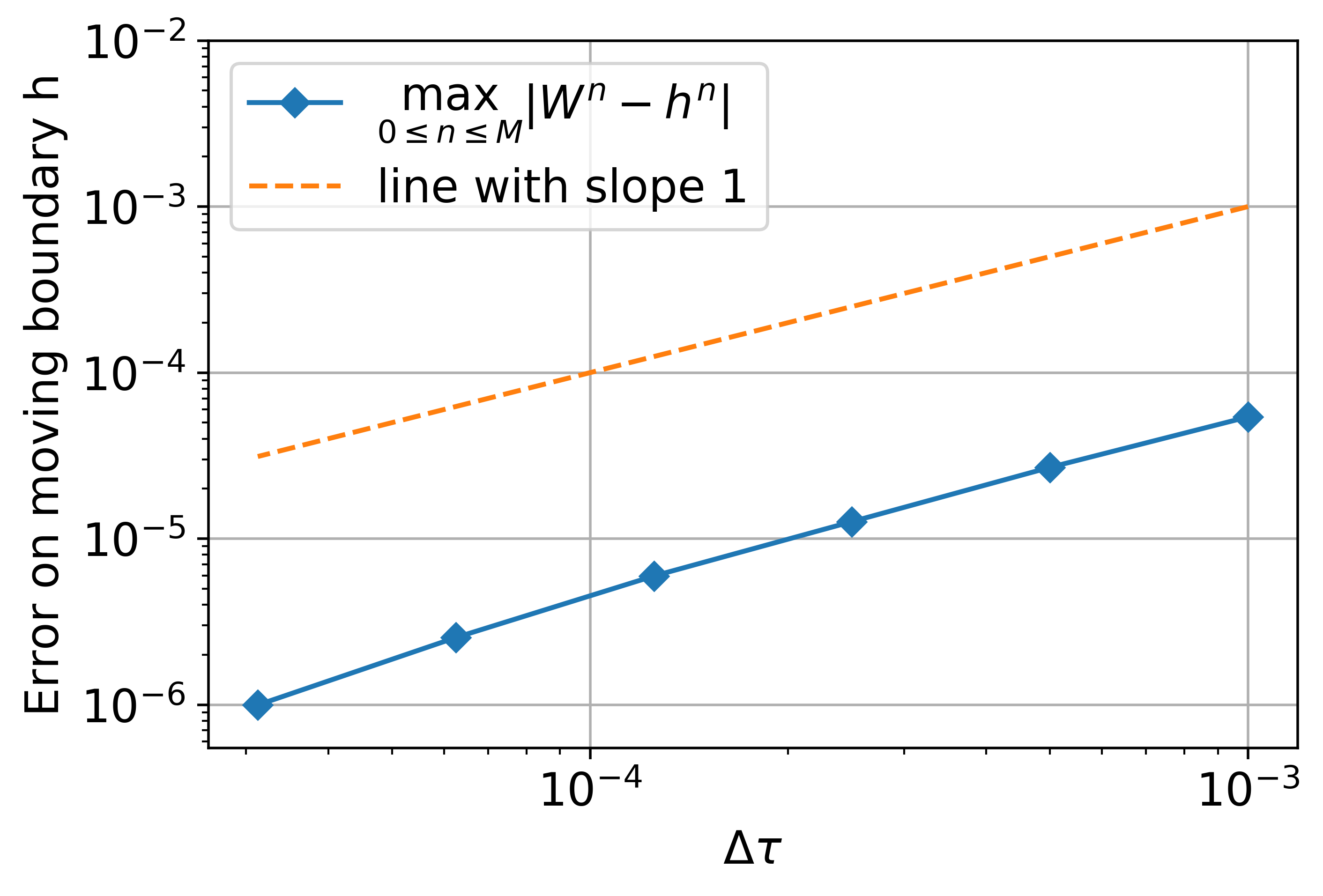}
	\hspace{0.1cm}
	\includegraphics[width=0.43\textwidth]{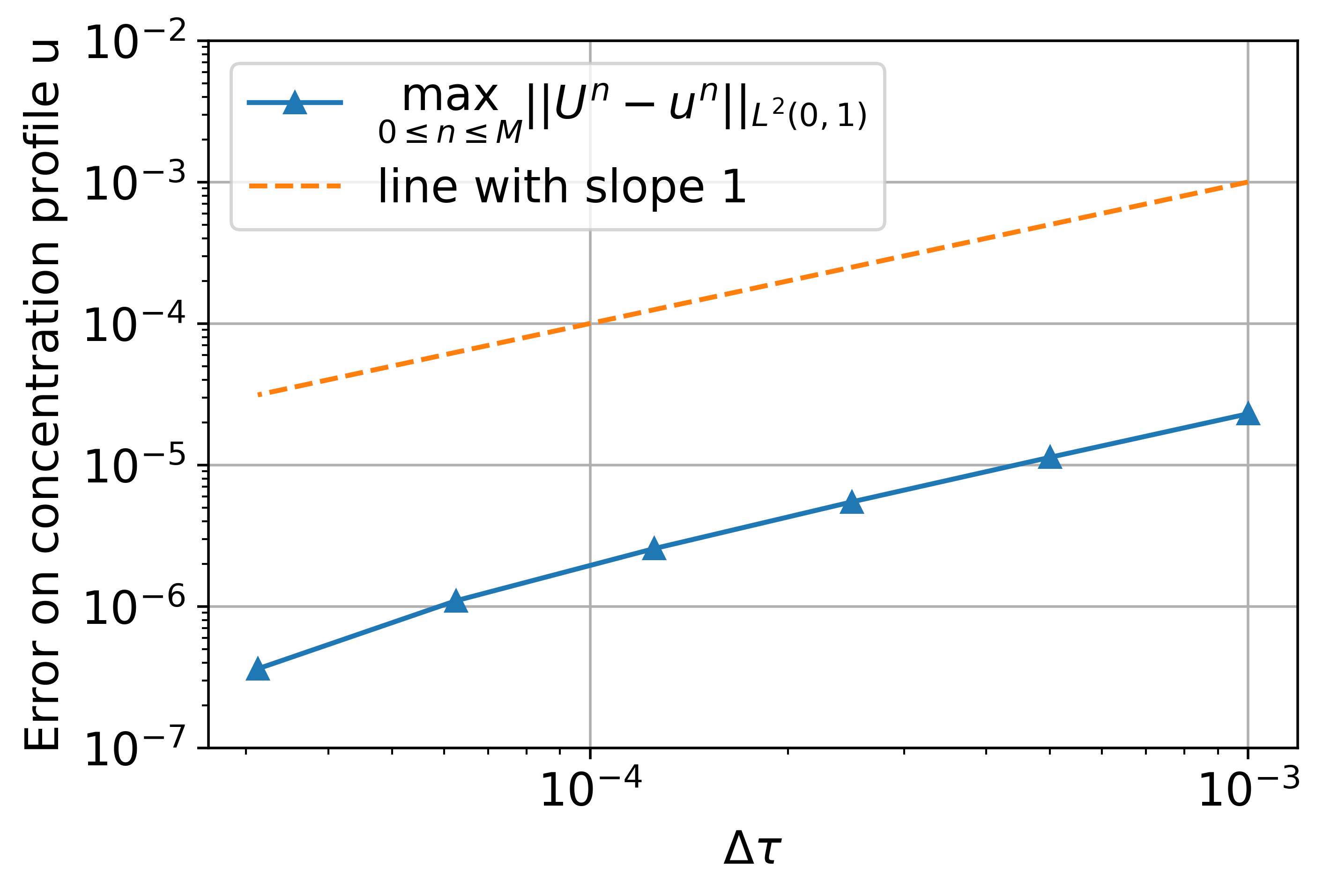}
	\caption{Convergence order in time when space mesh size is fixed with $N = 320$. Dash lines are lines of slope 1.  Left: Log log scale plot of error on the boundary $\max_{0\leq n \leq M}|W^n - h^n|$.  Right: Log log scale plot of error on the concentration $\max_{0\leq n \leq M}\|U^n - u^n\|_{L^2(0, 1)}$.}
	\label{PlotlogLogfullyTime}
\end{figure}
These numerical results are in agreement with the convergence orders proven in Section \ref{fully}.

\section{Conclusion}\label{conclusion}


We shown a fully discrete scheme for the numerical approximation of a moving boundary problem describing diffusants penetration into rubber.
The proposed scheme utilizes the Galerkin finite-element method in space and the backward Euler method in time.
By using Brouwer's fixed-point theorem, we were able to prove the existence of solution to the fully discrete problem. As main result, we obtained  \textit{a priori} error estimates for the mass concentration of diffusants as well as for the position of the moving boundary.  
 The convergence turns to be of first order in both space and time  for the approximation of the   mass concentration of diffusants as well as for the approximation of the position of the moving boundary. Finally, we illustrated numerically the order of convergence in space and time to confirm  the  theoretically obtained results.  It could be that the order of convergence in time can be improved by selecting other, perhaps better suited, time discretization schemes than the backward Euler one.  Because of the presence of the moving boundary, the convergence order is space is lower than what we would expect  for the finite element approximation of standard linear parabolic problems. This is in agreement with what is stated in the literature concerning the numerical approximation of one-dimensional moving-boundary problems.

\section*{Acknowledgments}
The activity of S.N.  and A.M. is financed partially by the Swedish Research Council's project "{\em  Homogenization and dimension reduction of thin heterogeneous layers}", grant nr. VR 2018-03648. A.M. also thanks the Knowledge Foundation for the grant KK 2019-0213, which led to the formulation of this problem setting. We benefited of fruitful discussions on closely related topics with T. Aiki (Tokyo), K. Kumazaki (Nagasaki), N. Kr\"oger (Hamburg), and U. Giese (Hannover). 

\begin{center}
	\bibliographystyle{plain}
	\bibliography{Fully_discrete_approximation_paper_3}
\end{center}
\end{document}